\documentclass[11pt]{amsart}
\usepackage{graphicx}
\usepackage[english]{babel}
\usepackage{amsthm}
\usepackage{float} 
\usepackage{amsmath}
\usepackage{amssymb}
\usepackage{amsfonts}
\usepackage{color}
\usepackage{algorithm}
\usepackage{algorithmic}
\usepackage{enumerate}
\usepackage{cite}
\usepackage{hyperref}
\usepackage{stmaryrd}
\usepackage{caption}
\usepackage{subcaption}
\usepackage{mathrsfs}
\usepackage{enumitem}
\usepackage{mathrsfs}  

\usepackage{geometry}

\numberwithin{equation}{section}

\allowdisplaybreaks[3]

\newtheorem{thm}{Theorem}[section]
\newtheorem{lemma}[thm]{Lemma}
\newtheorem{assumption}[thm]{Assumption}

\newtheorem{cor}[thm]{Corollary}
\newtheorem{prop}[thm]{Proposition}
\theoremstyle{definition}
\newtheorem{definition}[thm]{Definition}
\theoremstyle{remark}
\newtheorem{remark}[thm]{Remark}

\usepackage{array}

\usepackage{blindtext}

\newcommand{\comment}[1]{}
\newcommand{\f}{{\bm f}}
\newcommand{\A}{\bA}
\newcommand{\T}{\bT}
\newcommand{\VV}{\mathscr{V}}
\newcommand{\QQ}{\mathscr{Q}}
\newcommand{\LL}{\mathfrak{L}}
\newcommand{\NN}{\mathfrak{N}}
\newcommand{\WW}{\mathfrak{W}}

\usepackage{bm}

\newcommand{\p}{\partial}
\newcommand{\Div}{{\rm div}\,}
\newcommand{\bcurl}{{\mathbf{curl}}\,}
\newcommand{\bL}{\bm L}
\newcommand{\bu}{\bm u}
\newcommand{\bw}{\bm w}
\newcommand{\bv}{\bm v}
\newcommand{\bn}{\bm n}
\newcommand{\bt}{\bm t}
\newcommand{\bH}{\bm H}
\newcommand{\bp}{\bm p}
\newcommand{\btau}{\bm \tau}

\newcommand{\bA}{\bm A}
\newcommand{\bT}{\bm T}

\newcommand{\bq}{\bm q}
\newcommand{\Ome}{\Omega}
\newcommand{\bkappa}{\bm \kappa}
\newcommand{\bsigma}{\bm \sigma}
\newcommand{\calT}{\mathcal{T}}
\newcommand{\bI}{\boldsymbol{I}}
\newcommand{\bPi}{\bm \Pi}
\newcommand{\bbP}{\mathbb{P}}

\newcommand{\pol}{\mathcal{P}}
\newcommand{\bpol}{\boldsymbol{\pol}}
\newcommand{\bpsi}{\bm \psi}

\newcommand{\bbR}{\mathbb{R}}
\newcommand{\bs}{\bm s}

\newcommand{\calS}{\mathcal{S}}
\newcommand{\bN}{\bm N}
\newcommand{\calV}{\mathcal{V}}
\newcommand{\bbeta}{\bm \beta}

\newcommand{\MJN}[1]{\textcolor{black}{#1}}
\newcommand{\MN}[1]{\textcolor{black}{#1}}
\newcommand{\SG}[1]{\textcolor{black}{#1}}

\title[Convergence of Lagrange FEMs for 3D Maxwell Eigenvalue Problem]{Convergence of Lagrange \MN{Finite Element Methods} for Maxwell Eigenvalue Problem in 3D}

\author[D. Boffi]{Daniele Boffi}
\address{University of Science and Technology, Saudi Arabia and
University of Pavia, Italy}
\email{daniele.boffi@kaust.edu.sa}

\author[S. Gong]{Sining Gong}
\address{Division of Applied Mathematics, Brown University, Providence, RI 02912 }
\email{sining\_gong@brown.edu}

\author[J. Guzm\'an]{Johnny Guzm\'an}
\address{Division of Applied Mathematics, Brown University, Providence, RI 02912 }
\email{johnny\_guzman@brown.edu}           
\thanks{The first author is partially supported by IMATI/CNR and by PRIN/MIUR. 
The third author was supported in part by NSF grant DMS--1913083.  The fourth
author was supported in part by NSF grant DMS--2011733.}

\author[M. Neilan]{Michael Neilan}   
\address{Department of Mathematics, University of Pittsburgh, Pittsburgh, PA 15260}
\email{neilan@pitt.edu}

\begin{document}

\maketitle

\begin{abstract}
    We  prove convergence of the Maxwell eigenvalue problem using quadratic or higher Lagrange finite elements on Worsey-Farin splits in three dimensions. To do this, we construct two Fortin-like operators to prove uniform convergence of the corresponding source problem. We present numerical experiments to illustrate the  theoretical results.  
\end{abstract}

\thispagestyle{empty}

\section{Introduction} 
It is well known that, in contrast to  N\'ed\'elec edge elements \cite{MFE1980}, the direct use of Lagrange finite elements \MN{fail} to approximate Maxwell's eigenvalue problem, 
as they lead to erroneous solutions on generic triangulations (see for example \cite{FE2006, FE2010}). 
However, Wong and Cendes \cite{wong1988combined} numerically show Lagrange \MN{finite element methods (FEMs)} give the correct approximations on certain meshes. In particular, in two dimensions they show that the use of linear Lagrange finite element spaces defined on Powell-Sabin \cite{powell1977piecewise} meshes 
lead to accurate approximations. Likewise, in \cite[Example 3]{wong1988combined} they demonstrate that the use of 
quadratic Lagrange finite element spaces on ``consistent  tetrahedral meshes'' in three dimensions lead to correct approximations of Maxwell's eigenvalue problem. Although Wong and Cendes do not explicitly define ``consistent tetrahedral meshes'', it is reasonable
to assume they are referring to the three-dimensional analogue of Powell-Sabin triangulations, in particular, 
Worsey-Farin  meshes \cite{worsey1987ann} which we recall below. 

Recently in \cite{CLM2021} we theoretically justified the numerical experiments of Wong and Cendes \cite{wong1988combined} in two dimensions. 
We  proved that indeed linear (and higher) Lagrange elements on Powell-Sabin triangulations yield
discrete eigenvalues that converge to the true eigenvalues as the mesh parameter tends to zero.
The theory also shows the convergence of discrete eigenvalues
using quadratic (and higher) Lagrange elements on Clough-Tocher splits \cite{clough1965finite},
as well as quartic (and higher) Lagrange elements 
 on general meshes without nearly singular vertices.  Similar to 
 the two families of N\'ed\'elec edge elements, the spaces mentioned above fit into a discrete de Rham sequence (see for example \cite{guzman2020exact}). To prove convergence of the eigenvalue problem is suffices to prove uniform estimates of the solution operator of the corresponding source problem. One main tool in two dimensions  was the construction of a Fortin-like operator \cite{CLM2021}.

The present paper can be considered a continuation of paper \cite{CLM2021},
where we consider Lagrange elements in three dimensions on Worsey-Farin triangulations 
on a contractible, polyhedral domain. 
We, again, mathematically justify the numerical experiments of Wong and Cendes \cite{wong1988combined} and prove that the use of 
quadratic (or higher) Lagrange elements on Worsey-Farin
splits lead to accurate approximations to Maxwell's eigenvalue 
problem. The present analysis in three dimension is more involved than the two dimensional analysis given in \cite{CLM2021}. In particular, we need to develop two Fortin-like operators whereas in \cite{CLM2021} only one was needed. To do this,  we exploit that Lagrange elements on Worsey-Farin splits 
also fit into an a discrete de Rham sequence 
\cite{Ex2020}. Then, using the degrees of freedom given in 
\cite{Ex2020} we  construct a Fortin-like  operators for both curl and divergence operators.  In order to prove that these operators are bounded, we 
use certain embeddings that hold on Lipshitz polyhedral domains; 
see Section \ref{embeddings}. 

To the best of our knowledge, this seems to be the first paper theoretically justifying convergence of Lagrange elements on simplicial meshes in three dimensions 
without modifying the bilinear form. In contrast, several papers prove convergence of Lagrange elements, where they add penalization or regularization terms to the bilinear form; see \cite{bonito2011approximation,badia2012nodal, buffa2009solving, du2020mixed, duan2019new, duan2019family}. Parallel work by Hu et al. \cite{christiansen2018generalized, hu2021spurious, hu2022partially} also develop finite elements on different splits with Lagrange or partially discontinuous elements while leaving the bilinear forms unchanged. In particular, in  \cite{hu2022partially}  partially discontinuous elements where applied to the Maxwell eigenvalue problem in three dimensions on Worsey-Farin splits and they show convergence numerically.

We provide numerical results confirming our theoretical findings. In particular, we show that one must use at least quadratic Lagrange elements for Worsey Farin splits to get convergence. That is, the use of linear Lagrange elements on Worsey-Farin refinements do not yield convergent
approximations. 

The paper is organized as follows. In the next section, 
we state the Maxwell eigenvalue problem and its mixed formulation.
We also introduce general primal and mixed finite element methods
for the eigenvalue problem and present a convergence framework.
In Section \ref{sec-prelims},
we give several preliminary results including trace inequalities and Sobolev embeddings.
Section \ref{sec-WF} gives the definition of Worsey-Farin triangulations,
summarizes some exactness properties of finite element spaces
on such meshes, and constructs a Scott-Zhang-type interpolant.
In Section \ref{sec-Fortin}, we construct two Fortin-like operators and show stability estimates for those two operators. 
As a byproduct, in Section \ref{apply-Max}, we show convergence of Lagrange finite element methods
for the Maxwell eigenvalue problem on Worsey-Farin triangulations provided
the polynomial degree is at least two.
Finally, in Section \ref{sec-numerics}, we present some numerical experiments illustrating that continuous piecewise polynomials can be applied to three dimensional Maxwell eigenvalue problem after Worsey-Farin refinement.

\section{The Maxwell's eigenvalue problem, its Discretization, and Convergence Framework}\label{sec-framework}

Let $\Omega \subset \mathbb{R}^3$ be a contractible, Lipschitz polyhedral domain and consider the eigenvalue problem: Find $\bu \in \bH_0(\bcurl, \Omega)$ such that
 \begin{equation} \label{equ:Ceig}
 (\bcurl \bu, \bcurl \bv) = \eta^2 (\bv,\bv), ~~~ \forall \bv \in \bH_0(\mathbf{curl}, \Omega),
 \end{equation}
 where $(\cdot,\cdot)$ is the $L^2(\Omega)$ inner product \MJN{over $\Omega$}, and
 \begin{equation*}
 \bH_0(\mathbf{curl}, \Omega):=\{\bv \in \bL^2(\Omega): \bcurl \bv \in \bL^2(\Omega) \text{~and~} \bv \times \bn=0 \text{~on~} \partial \Omega\}.
 \end{equation*}
 Here, $\bn$ is an exterior unit normal vector on $\partial \Omega$. 

 Accordingly, a \MJN{canonical} finite element method with 
 \MJN{respect to} a given finite-dimensional space $\VV_h \subset \bH_0(\mathbf{curl},\Omega)$ is to find 
 $\bu_h \in \VV_h \backslash \{0\}$ and $\eta_h \in \mathbb{R}$ such that
 \begin{equation} \label{equ:Deig}
 (\bcurl \bu_h, \bcurl \bv_h) = \eta_h^2 (\bu_h,\bv_h), ~~~ \forall \bv_h \in \VV_h.
 \end{equation}

\SG{When $\eta \neq 0$}, a mixed formulation given by Boffi et al. \cite{CME1999}  is equivalent to \eqref{equ:Ceig} : find $\lambda \in \mathbb{R} \backslash \{0\}$ and $0 \neq \bp \in \bH_0({\rm div}^0, \Omega)$, $\bsigma \in \bH_0(\mathbf{curl}, \Omega)$ such that:
\begin{equation} \label{equ:CMod}
    \begin{aligned}
        (\bsigma,\btau)+(\bp,\bcurl\btau)&=0~~~~~~~~~& \forall\btau \in \bH_0(\mathbf{curl}, \Omega), \\
        (\bcurl \bsigma,\bq)&=-\lambda(\bp,\bq)~~~~~~ & \forall \bq \in \bH_0({\rm div}^0, \Omega),
    \end{aligned}
\end{equation}
where 
\begin{alignat*}{1}
\bH_0({\rm div}, \Omega):=&\{\bv \in \bL^2(\Omega): {\rm div} \bv \in L^2(\Omega),  \bv \cdot \bn=0 \text{~on~} \partial \Omega\}, \\    
\bH_0({\rm div}^0, \Omega):=&  \{\bv \in \bH_0({\rm div}, \Omega): \Div \bv=0\}.
\end{alignat*}
We note that $\bH_0({\rm div}^0, \Omega)=\bcurl \bH_0(\bcurl, \Omega)$, and also $\lambda= \eta^2$, $\bsigma= \bu$ and $\bp= -\frac{\bcurl \bu}{\lambda}$. 

An equivalent mixed formulation of \MJN{the discrete problem} \eqref{equ:Deig} (\SG{when $\eta_h \neq 0$}) is: find $\lambda_h \in \mathbb{R}\backslash \{0\}$ and $0 \neq \bp_h \in \QQ_h, \bsigma_h \in \VV_h$ such that:
\begin{equation} \label{equ:DMod}
    \begin{aligned}
        (\bsigma_h,\btau_h)+(\bp_h,\bcurl\btau_h)&=0~~~~~~~~~&\forall\btau_h \in \VV_h, \\
        (\bcurl \bsigma_h,\bq_h)&=-\lambda_h(\bp_h,\bq_h)~~~~~~ &\forall \bq_h \in \QQ_h,
    \end{aligned}
\end{equation}
where $\QQ_h:= \bcurl \VV_h$. Analogous to the continuous setting, 
we have $\lambda_h= \eta_h^2$, $\bsigma_h= \bu_h$ and $\bp_h= -\frac{\bcurl \bu_h}{\lambda_h}$.

We follow the classical theory (e.g., \cite[Section 14]{FE2010}) and analyze the \MJN{mixed finite element} method 
\MJN{\eqref{equ:DMod}}
by considering the corresponding source problem. Define the solution operators $\bA: \bL^2(\Omega) \to \bH_0(\mathbf{curl}, \Omega)$ and $\bT:\bL^2(\Omega) \to \bH_0({\rm div}^0, \Omega)$ such that for given ${\bm f} \in \bL^2(\Omega)$, there holds
\begin{equation}\label{CSour}
    \begin{aligned}
        (\boldsymbol{Af},\btau)+(\boldsymbol{Tf},\bcurl\btau)&=0~~~~~~~~~& \forall\btau \in \bH_0(\mathbf{curl}, \Omega), \\
        (\bcurl \boldsymbol{Af},\bq)&=({\bm f},\bq)~~~~~~ & \forall \bq \in \bH_0({\rm div}^0, \Omega).
    \end{aligned}
\end{equation}
Likewise, the discrete solution operators $\bA_h:\bL^2(\Omega)\to \VV_h$
and $\bT_h:\bL^2(\Omega)\to \QQ_h$
are defined as
\begin{equation}\label{DSour}
    \begin{split}
        (\bA_h{\bm f},\btau_h)+(\bT_h {\bm f},\bcurl\btau_h)=0~~~~~~~~~& \forall\btau_h \in \VV_h, \\
        (\bcurl \bA_h{\bm f},\bq_h)=({\bm f},\bq_h)~~~~~~ & \forall \bq_h \in \QQ_h.
    \end{split}
\end{equation}

\subsection{Convergence theory}\label{subsec-conv}
It is well known that the convergence of the eigenvalues
to the discrete problem \eqref{equ:DMod} converge
to the exact eigenvalues (given in problem \eqref{equ:CMod})
provided the source problem converges uniformly
(see for example \cite[Section 7]{FE2010}).  In the next proposition, the operator norm is defined as
\begin{equation}
    \|\bT\| := \sup\limits_{\f \in \SG{\bL^2}(\Omega) \backslash \{0\}} \frac{\|\bT \f\|_{L^2(\Omega)}}{\|\f\|_{L^2(\Omega)}}.
\end{equation}
\begin{prop}\label{prop1}
Let $\T$ and $\T_h$ be defined from \eqref{CSour} and \eqref{DSour}, respectively, and suppose that $\|\T-\T_h\| \rightarrow 0$ as $h \rightarrow 0$. 
Consider the problem \eqref{equ:CMod} with the nonzero eigenvalues  $0<\lambda^{(1)} \leq \lambda^{(2)} \leq \cdots$ and the problem (\ref{equ:DMod}) 
with the nonzero eigenvalues $0<\lambda_h^{(1)} \leq \lambda_h^{(2)} \leq \cdots$. Then, for any fixed $i$, $\lim_{h \rightarrow 0} \lambda_h^{(i)} = \lambda^{(i)}$.
\end{prop}

It will suffice to verify one assumption of our discrete spaces to guarantee $\|\T-\T_h\| \rightarrow 0$. To describe this assumption,  we introduce the space
\begin{equation} \label{vqh}
    \VV^t(\QQ_h)=\{\btau \in \VV^t: \bcurl \tau \in \QQ_h\}.
\end{equation}

\begin{assumption}\label{VQ}
We assume the existence of a projection
$\bPi_{\VV}:\VV^t(\QQ_h) \to \VV_h $ such that 
\begin{alignat*}{2}
    \bcurl\bPi_{\VV} ~\btau &= \bcurl\btau ~~~~~~~ &&\forall\btau \in \VV^t(\QQ_h),\\
    \|\bPi_{\VV}\btau-\btau\| &\leq \omega_0(h) (\|\btau\|_{H^{1/2+\delta}(\Omega)} + \|\bcurl\btau\|_{L^2(\Omega)} )~~~~~~~ &&\forall\btau \in \VV^t(\QQ_h).
\end{alignat*}
Furthermore, we assume that the $L^2$-orthogonal projection $\bbP_Q:\bL^2(\Omega) \to \QQ_h$ satisfies
\begin{equation}
    \|\bbP_Q \bp-\bp\|_{L^2(\Omega)} \leq \omega_1(h)\|\bcurl \bp\|_{L^2(\Omega)}~~~~~\forall \bp \in \bH(\mathbf{curl}, \Omega) \cap \bH_0({\rm div}^0, \Omega).
\end{equation}
Here, $\omega_i > 0$ satisfies $\lim_{h\rightarrow0^+} \omega_i(h)=0$ for $i=0,1$.
\end{assumption}

With this assumption we have the following result. We give the proof of this result in the appendix; it is  similar to the argument in
\cite{CLM2021}.
\begin{thm} \label{thm:Th}
Suppose that $(\VV_h,\QQ_h)$ satisfy Assumption \ref{VQ} and let $\bT$ and $\bT_h$ defined as in the equations (\ref{CSour}) and (\ref{DSour}), respectively. Then we have
\begin{equation}
    \|\bT-\bT_h\| \leq C(\omega_0(h)+\omega_1(h)).
\end{equation}
\end{thm}
 The following corollary is a consequence of the above theorem and Proposition \ref{prop1}.
\begin{cor}
Consider the problem \eqref{equ:CMod} with the nonzero eigenvalues  $0<\lambda^{(1)} \leq \lambda^{(2)} \leq \cdots$ and the problem (\ref{equ:DMod}) 
with the nonzero eigenvalues $0<\lambda_h^{(1)} \leq \lambda_h^{(2)} \leq \cdots$. Suppose that $(\VV_h,\QQ_h)$ satisfy Assumption \ref{VQ}. Then, for any fixed $i$, $\lim_{h \rightarrow 0} \lambda_h^{(i)} = \lambda^{(i)}$.
\end{cor}

\section{Preliminaries}\label{sec-prelims}

\subsection{Trace Theorems and Inverse inequalities}
Here we state inequalities that allow us to estimate the Fortin-type projections. We start by recalling the definition of fractional-order Sobolev spaces
and their accompanying (semi-) norm; see for example \cite{Elliptic2011}.
\begin{definition}
Let $U \subset \bbR^d$ be an open set. For $0<s<1$ and $1<p<\infty$
define
\[ W^{s,p}(U):= \{ u \in L^p(U): \int_{U} \int_{U} \frac{|u(x)-u(y)|^p}{|x-y|^{d+sp}}\,dx\,dy < \infty \}.
\]
The accompanying semi-norm and norm on $W^{s,p}(U)$ are given respectively by
\[ |u|_{W^{s,p}(U)} := (\int_{U} \int_{U} \frac{|u(x)-u(y)|^p}{|x-y|^{d+sp}}\,dx\,dy)^{1/p};
\]
\[ \|u\|_{W^{s,p}(U)} :=(\|u\|^p_{L^p(U)} + |u|_{W^{s,p}(U)}^p)^{1/p}.
\]
Finally, we use the notation $H^s(U):= W^{s,2}(U)$.
\end{definition}
We also extend the definition of fractional-order Sobolev norms in the case $U= \partial S$, where $S$ is a simplex. We take the same definition as above, where we view $U$ as a $d$-dimensional manifold; see \cite{Elliptic2011}.

We will use the following basic Sobolev embedding result; see \cite[Theorem 1.4.4.1]{Elliptic2011}. We restrict ourselves to a simplex for simplicity. 
\begin{lemma}
Let $K$ be a $d$-dimensional simplex. Then, 
\begin{equation}\label{sobolevembedding}
\|w\|_{W^{t,q}(K)} \le C \| w\|_{W^{s,p}(K)}, \quad 0\le t \le s \le 1,\  1\le p \le q<\infty,  
\end{equation}
with $s-d/p=t-d/q$. The constant $C>0$ depends on $K$. 
\end{lemma}

We will also need a trace inequality; see \cite[Theorem 1.5.1.2]{Elliptic2011} 
\begin{lemma}
Let $K$ be a $d$-dimensional simplex. Let  $0 < s-1/p < 1, 0 \le s \le 1$. If $w \in W^{s,p}(K)$, then  
\begin{equation}\label{traceinq} 
\|w\|_{W^{s-1/p,p}(\partial K)} \le C \| w\|_{W^{s,p}(K)}.  
\end{equation}
Moreover, if $v \in W^{s-1/p,p}(\partial K)$ there exists $w  \in W^{s,p}(K)$ such that 
$w|_{\p K} = v$ and
\begin{equation}\label{righttraceinq} 
 \| w\|_{W^{s,p}(K)} \le C \|v\|_{W^{s-1/p,p}(\partial K)}.      
\end{equation}
The constant $C>0$ depends on $K$.
\end{lemma}

We will often require particular inverse estimates for polynomial spaces, which follow from equivalence of norms on finite dimensional spaces; 
for more general inverse estimates consult, for example, \cite[Section 1.6 and Section 4.5]{brenner2008mathematical}.
\begin{lemma} \label{lem:lp}
Let $K$ be a $d$-dimensional simplex. For any $v \in \pol_{r}(K)$ the following bounds hold 
\begin{subequations} \label{equ:inv}
\begin{alignat}{2}
|v|_{W^{k,p}(K)} \leq & C h_K^{\ell-k+d/p-d/q} |v|_{W^{\ell,q}(K)} \quad \forall 1 \le  p, q \le \infty, 0 \le  \ell \le k \le 1, \label{inv1} \\
\|v\|_{L^{p}(\partial K)} \leq & C h_K^{-1/p} \|v\|_{L^p(K)} \quad \forall 1 \le  p \le \infty, \label{inv2}
\end{alignat}
where $C>0$ depends on the shape-regularity of $K$, $r$, and $d$, but is independent of $h_K$.
\end{subequations}
\end{lemma}

\subsection{Embeddings}\label{embeddings}
To prove convergence to the solution operator, we utilize 
certain embeddings of vector-valued functions.  To describe the results we introduce
the following space notation:
\begin{alignat*}{1}
\VV:=&\bH(\mathbf{curl}, \Omega) \cap \bH({\rm div}, \Omega), \\
\VV^t :=& \bH_0(\mathbf{curl}, \Omega) \cap \bH({\rm div}, \Omega),\\
\VV^n := & \bH(\mathbf{curl}, \Omega) \cap \bH_0({\rm div}, \Omega).
\end{alignat*}

We start with an embedding result given in \cite[Proposition 3.7]{VP1998}. 
\begin{prop} \label{emb}
If $\Omega$ is a Lipschitz polyhedron, there exists $\delta \in (0,\frac{1}{2}]$ and $C > 0$ such that:
\begin{equation}\label{eqn:emb123}
    \|\bv\|_{H^{1/2+\delta}(\Omega)} 
    \leq C(\|\bv\|_{L^2(\Omega)}+\|\bcurl\bv\|_{L^2(\Omega)}+\|\Div \bv\|_{L^2(\Omega)})~~~~~ \forall \bv \in \VV^t \cup \VV^n.
\end{equation}
\end{prop}
Here, the constant $\delta$ in \eqref{eqn:emb123}
may differ for $\bv\in \VV^t$ and $\bv\in \VV^n$.
Thus, we choose the smaller constant $\delta$ of these
two embeddings.
Next, we use a result given in \cite[Theorem 2.2]{FE2006} where we use that $\Omega$ is a contractible, Lipschitz polyhedron.
\begin{lemma}\label{lemmaemb}
There exists a positive constant $C$ such that for all $\bv \in \VV^t \cup \VV^n$, 
\begin{equation*}
    \|\bv\|_{L^2(\Omega)} \leq C (\|\bcurl\bv\|_{L^2(\Omega)}+\|\Div \bv\|_{L^2(\Omega)}).
\end{equation*}
\end{lemma}


 \section{Finite element spaces on Worsey-Farin splits}\label{sec-WF}
 \subsection{Definitions and Notations}
 For a set of simplices $\calS$, we use $\Delta_s(\calS)$ to denote the set of 
$s$-dimensional simplices ($s$-simplices for short) in $\calS$.  
If $\calS$ is a simplicial triangulation of a domain $D$ with boundary, then $\Delta_s^I(\calS)$ denotes the 
subset of $\Delta_s(\calS)$ that \SG{does} not belong to the boundary
of the domain.  If $S$ is a simplex, then we use the convention $\Delta_s(S) = \Delta_s(\{S\})$.
 For a non-negative integer $r$, we use $\mathcal{P}_r(S)$ to denote the space of piecewise polynomials of degree $\leq r$ on $S$,
 and we define
 \begin{alignat*}{2}
& \pol_r(\calS) = \prod_{S\in \calS} \pol_r(S),\quad &&\pol_r^c(\calS) = \pol_r(\calS)\cap C^0(D),\quad \text{with }D = \bigcup_{S\in \calS} \bar S,\\
&\mathring{\pol}_r(\calS) = \pol_r(\calS)\cap L^2_0(D),\quad && \mathring{\pol}_r^c(\calS) = \{v\in \pol_r(\calS):\ v|_{\p D} = 0\},
 \end{alignat*}
where $L^2_0(D)$ is the space of square integrable functions
with vanishing mean.  Analogous vector-valued spaces are denoted in boldface, e.g., $\bpol_r^c(\calS)=[\pol_r^c(\calS)]^3$.

 Given a family of shape-regular, simplicial triangulations 
 $\{ \mathcal{T}_h\}$ of $\Omega$, let $h_T=$ diam$ (T)$, $\forall T \in \mathcal{T}_h$ and $h=\max_{T \in \mathcal{T}_h}h_T$. Since the meshes are shape regular, there exists a constant $c_0 >0 $ such that 
 \begin{equation}\label{def:calKDef} 
  h_T \le c_0 \rho_T, \quad \rho_T = \text{diameter of largest sphere contained in } \Bar{T}, 
 \end{equation}
 for all $T \in \mathcal{T}_h$. 
 We now describe the construction of a Worsey-Farin triangulation \cite{worsey1987ann} from the original triangulation $\calT_h$.  
 
 For an arbitrary tetrahedra $T$, we first show how to obtain 
 a local Worsey-Farin triangulation of $T$, denoted by  $T^{wf}$.  
 This is done via the following two steps (cf.~\cite[Section 2]{Ex2020})):
 \begin{enumerate}
\item  Connect the incenter $z_T$ of of $T$
 to its (four) vertices. 
\item For each face $F$ of $T$ choose $m_F \in  {\rm int}(F)$.  We then connect $m_F$ to the three vertices of $F$ and to the incenter $z_T$.
\end{enumerate} 
\noindent \SG{Here, ${\rm int}(F)$ denotes the interior of $F$.} This procedure divides $T$ into the twelve 
tetrahedra which define $T^{wf}$.

 To obtain a Worsey-Farin refinement $\mathcal{T}^{wf}_h$ of a triangulation $\mathcal{T}_h$ we split each $T \in \mathcal{T}_h$ by the above procedure. However, special care is needed in the choice of the point $m_F$. 
 For each interior face $F = \overline{T_1} \cap \overline{T_2}$ with $T_1$,
 $T_2 \in \calT_h$, let $m_F = L \cap F$ where $L = [{z_{T_1},z_{T_2}}]$, the line segment connecting the incenters of $T_1$ and $T_2$. The fact that such a $m_F$ exists is established in \cite[Lemma  16.24]{lai2007spline}. For a boundary face $F$ with $F = \overline{T} \cap \p\Omega$ with $T\in \calT_h$, let $m_F$ be the barycenter of $F$. In \cite{lai2007spline} it is conjectured that the resulting triangulation is shape regular. We will assume throughout that $\{\mathcal{T}^{wf}_h\}$ is shape regular with the regularity constant  related to the shape regularity constant of $\{\mathcal{T}_h\}$.  

For any $F \in \Delta_2(T)$, we see that the refinement $T^{wf}$ induces a Clough-Tocher triangulation of $F$, i.e.,
a triangulation consisting of three triangles, each having
the common vertex $m_F$; we denote this set by $F^{ct}$.
Let $e_F \in \Delta_1^I(F^{ct})$ be an arbitrary, but fixed internal edge of $F^{ct}$. 
Then further define 
\[
\mathcal{E}(\mathcal{T}_h^{wf}) = \{e \in \Delta_1^I(F^{ct}): \forall F \in \Delta_2^I(\mathcal{T}_h) \}.
\] 

Let $T_1$ and $T_2$ be adjacent tetrahedra in $\calT_h$ that share a face $F$.
Write  $F = [x_0,x_1,x_2]$ where $\{x_i\}_{i=0}^2$ are the vertices of $F$, and similarly
write $T_i = [a_i,x_0,x_1,x_2]$ for $i = 1,2$, where $a_i$ is the vertex of $T_i$ not shared by $F$.
 Let $z_{T_i}$ be the incenter of $T_i$ and set $K_i = [z_{T_i},x_0,x_1,x_2]$, $i = 1,2$.  
 We denote by $K^{wf}_i$  the triangulation 
$\calT_h^{wf}$ restricted to $K_i$, that is, $K^{wf}_i$ consists
of three tetrahedra which are each in $\bar T_i$ and have a single face that lie in $\bar F$.

Let $e \in \Delta_1^I(F^{ct})$ be one of the three internal edges in the Clough-Tocher refinement of $F$,
and let $\{K_i^j:\ j=1,2\}\subset K_i^{wf}$ be the two tetrahedra in $K_i^{wf}$
that have $e$ as an edge ($i=1,2$).  We assume that $K_i^j$ is labeled such that
$K_1^2$ and $K_2^1$ share a common face.
We then define
 \[
 \theta_e(p) = p|_{K_1^1}-p|_{K_1^2}+p|_{K_2^1}-p|_{K_2^2},~~~~~~ \text{on~} e.
 \]

Let $T\in \calT_h$, $F\in \Delta_2(\calT_h)$ with $F\subset \p T$,
and $e\in \Delta_1^I(F^{ct})$.  
Denote by $\bt_e$ be the unit vector tangent to $e$ pointing away from $m_F$,
and let $\bn_F$ be the outward unit normal of $T$ restricted to $F$. 
Then there exist triangles $Q_1$, $Q_2 \subset F^{ct}$ such that $e = \partial Q_1 \cap \partial Q_2$. 
The jump of a piecewise smooth function $p$ across $e$ is defined as
\[
\llbracket p \rrbracket_e = (p|_{Q_1}-p|_{Q_2})\bs,
\]
where $\bs = \bn_F \times \bt_e$ is a unit vector orthogonal to $\bt_e$ and $\bn_F$. 

Let $\bp$ and $u$ be a smooth vector and scalar valued functions, respectively, on $T \in \calT_h$, and let $F \in \Delta_2(T)$. 
Then the tangential part of $\bp$ and $u$ are given by $\bp_F:= \bn \times \bp \times \bn$ and $u_F=u|_F$, respectively, where $\bn$ is the 
unit outward normal of $\partial T$. We have the following identities:
 \begin{equation*}
 \begin{aligned}
     \mathbf{curl}_F\, \bp_F = \mathbf{curl}\,\bp \cdot \bn,~~~~~ \mathbf{grad}_F\,u_F = \bn \times (\mathbf{grad}\,u \times \bn), ~~~~~ \text{on}~F.
 \end{aligned}    
 \end{equation*}

  \begin{figure}
      \centering
      \begin{subfigure}[b]{0.3\textwidth}
          \centering
          \includegraphics[width=\textwidth]{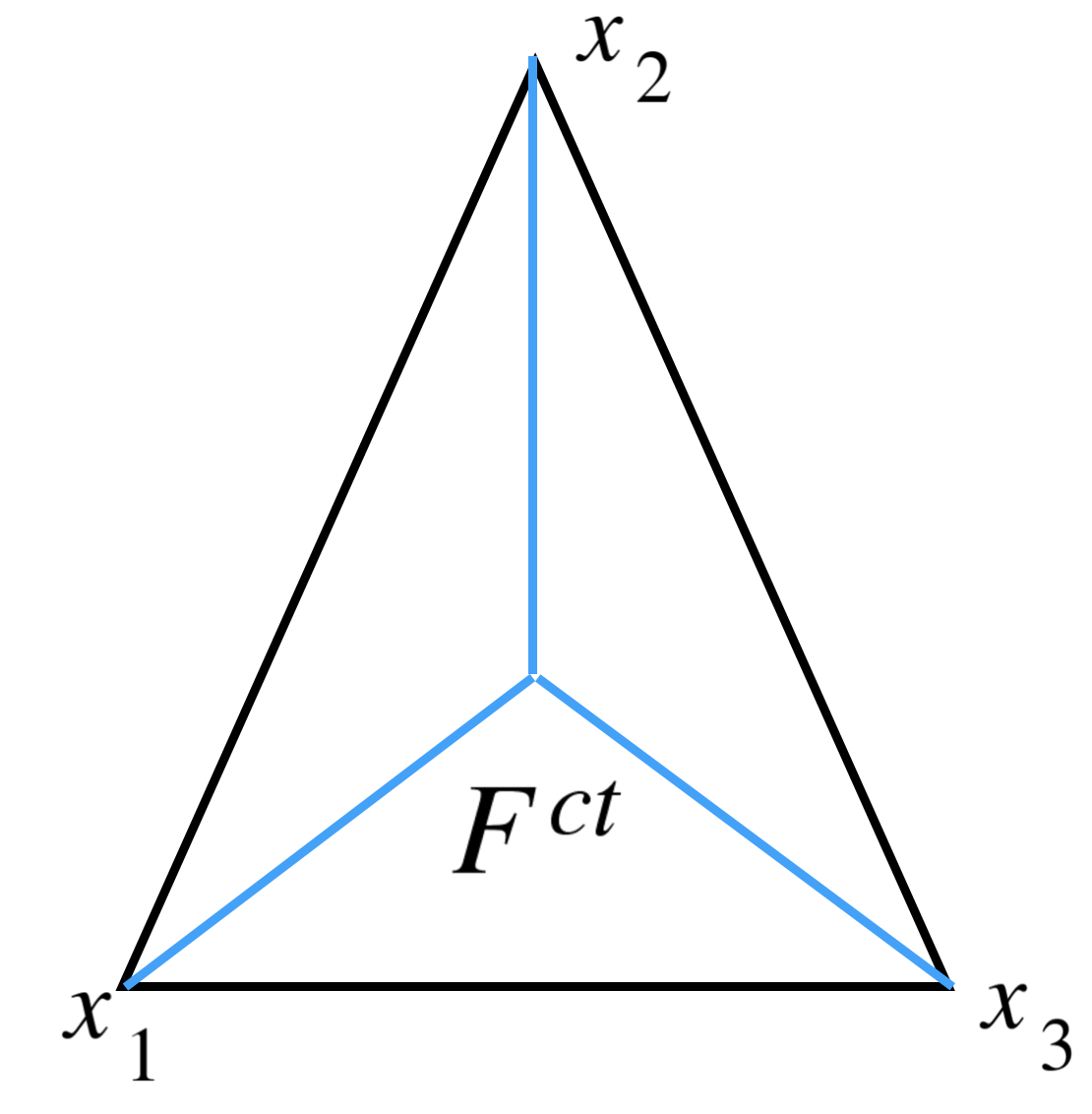}
          \caption{A representation of $F^{ct}$ and $\Delta_1^I(F^{ct})$ (indicated in blue).}
          \label{fig:Fct}
      \end{subfigure}
      \hfill
      \begin{subfigure}[b]{0.3\textwidth}
          \centering
          \includegraphics[width=\textwidth]{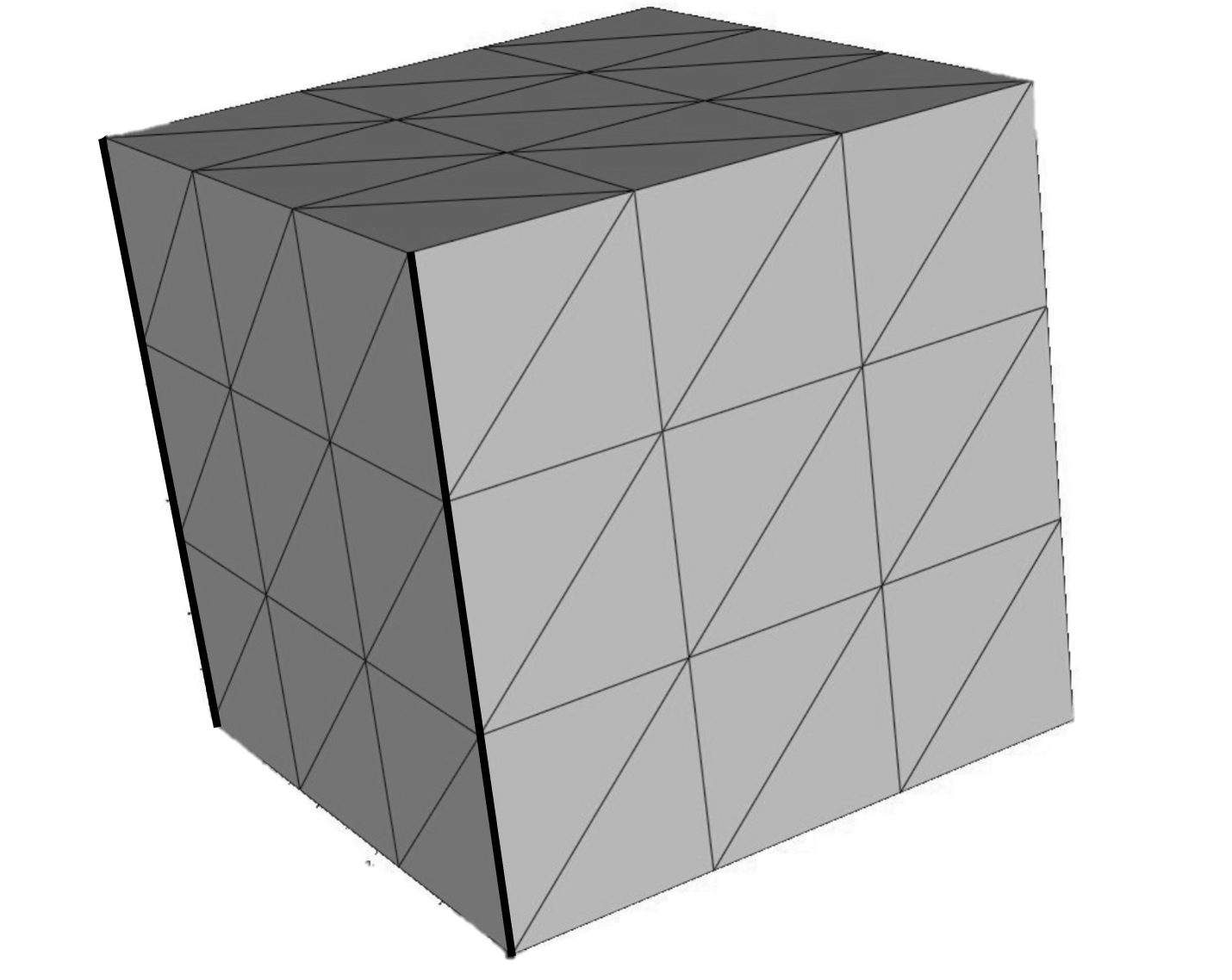}
          \caption{The original triangulation}
          \label{fig:Macro}
      \end{subfigure}
      \hfill
      \begin{subfigure}[b]{0.3\textwidth}
          \centering
          \includegraphics[width=\textwidth]{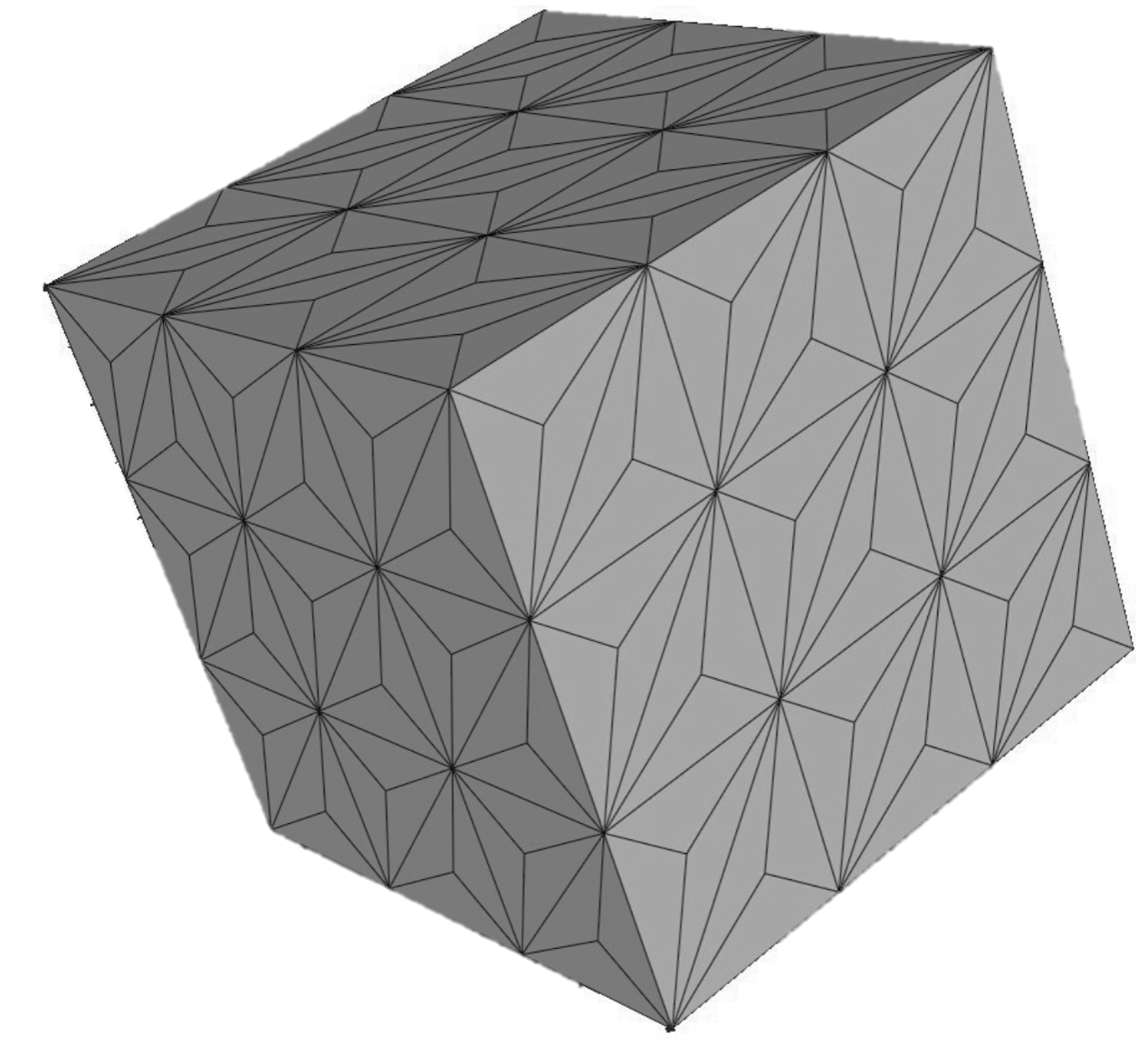}
          \caption{Worsey-Farin refinement}
          \label{fig:WF}
      \end{subfigure}
         \caption{The Worsey-Farin Splits}
         \label{fig:three graphs}
 \end{figure}
 
  We  define the function spaces with respect to a Clough-Tocher triangulation of a face $F$ (cf.~\cite[Section 3]{Ex2020}),
 \begin{equation*}
     \begin{split}
          & S_r(F^{ct}):=\{v \in \pol^c_{r}(F^{ct}): \mathbf{grad}_F\, v \in \boldsymbol{\mathcal{P}}_r^c(F^{ct})\}, \\
         & \mathcal{R}_{r}(F^{ct}) := S_r(F^{ct})\cap \mathring{\pol}_r^c(F^{ct}),\qquad
         \mathring{S}_r(F^{ct}):= \{v\in \mathcal{R}_r(F^{ct}):\ \mathbf{grad}_F\, v|_{\p F}=0\},
     \end{split}
 \end{equation*}
and define function spaces with respect to the Worsey-Farin triangulation of a tetrahedron $T\in \calT_h$, 
  \begin{equation*}
     \begin{aligned}
         & S_{r}(T^{wf}):= \mathcal{P}_r(T^{wf})\cap H^2(T),\qquad 
         && \mathring{S}_{r}(T^{wf}):=S_{r}(T^{wf})\cap H^2_0(T),\\
         & \bN_{r}(T^{wf}) := \bpol_{r}(T^{wf}) \cap \bH({\rm div}, T).
     \end{aligned}
 \end{equation*}

\subsection{Global spaces and degrees of freedom}
We define the global spaces (cf.~\cite[Section 6]{Ex2020})
 \begin{equation*}
     \begin{aligned}
         & \mathfrak{S}_{r} = \{v \in C^1(\Omega): q|_T \in S_r(T^{wf}), \forall T \in \mathcal{T}_h  \}, \\
         & \LL_{r-1} = \bpol_{r-1}^c(\calT_h^{wf}),\\ 
         & \NN_{r-2} = \{\bv \in \bH({\rm div}, \Omega): \bv|_T \in \bN_{r-2}(T^{wf}),~\forall T \in \mathcal{T}_h, \theta_e(\bv \cdot \bt_e)=0 ~\forall e \in \mathcal{E}(\mathcal{T}_h^{wf}) \},\\
         & \WW_{r-3} = \pol_{r-3}(\calT_h^{wf}). 
     \end{aligned}
 \end{equation*}
 The \MJN{definitions of the finite element spaces
 show the following sequence forms a complex (cf.~\cite{Ex2020})}
  \begin{equation} \label{equ:exact}
     \mathbb{R}  \xrightarrow{\subset} \mathfrak{S}_{r} \xrightarrow{\mathbf{grad}} \LL_{r-1} \xrightarrow{\mathbf{curl}} \NN_{r-2} \xrightarrow{\Div} \WW_{r-3} \rightarrow 0.
 \end{equation}
We also consider the complex with boundary conditions:
  \begin{equation} \label{equ:bspace}
     0 \rightarrow \mathfrak{S}_{r}^0 \xrightarrow{\mathbf{grad}} \LL^t_{r-1} \xrightarrow{\mathbf{curl}} \NN^n_{r-2} \xrightarrow{\Div} \WW^0_{r-3} \rightarrow 0,
 \end{equation}
 where 
 \begin{equation} \label{def:bspace}
 \begin{aligned}
    & \mathfrak{S}_{r}^0 = \mathfrak{S}_r\cap H^1_0(\Omega),\quad
     && \LL^t_{r-1} = \LL_{r-1}\cap \bH_0(\mathbf{curl},\Omega),\\
     & \NN^n_{r-2} = \NN_{r-2}\cap \bH_0({\rm div},\Omega),\quad
     && \WW^0_{r-3} = \WW_{r-3} \cap L^2_0(\Omega).
 \end{aligned}
 \end{equation}
 
 The following lemma summarizes the degrees of freedom (dofs) for $\LL_{r-1}$, $\NN_{r-2}$, and $\WW_{r-3}$ given in 
 \cite[\MJN{Lemmas 5.4--5.6}]{Ex2020}.
 \begin{lemma}\label{lem:LLDOFs}
 Let $r\ge 3$.
 \begin{enumerate}[leftmargin=0.75\parindent]
\item A function $\btau \in \bpol^c_{r-1}(T^{wf})$ is uniquely defined by the following conditions:
 \begin{subequations} \label{equ:dofL}
    \begin{alignat}{2}
      \label{dofL1}
          &\btau(a)~~~~~~~~~~~&&\forall a \in \Delta_0(T) \\
           \label{dofL2}
          & \int_e\btau \cdot {\bm \kappa} \, ds ~~~~~~~&&\forall {\bm \kappa} \in \bpol_{r-3}(e),~\forall e \in \Delta_1(T),\\
           \label{dofL3}
           & \int_e \llbracket \mathbf{curl} ~\btau \cdot \bt_e \rrbracket_e \kappa \, ds ~~~~~~~&&\forall \kappa \in \mathcal{P}_{r-3}(e),~\forall e \in \Delta_1^{I}(F^{ct})\backslash \{e_F\},~\forall F \in \Delta_2(T),\\
            \label{dofL4}
           & \int_{e_F} \llbracket \mathbf{curl} ~\btau \cdot \bt_{e_F} \rrbracket_{e_F} \kappa \, ds ~~~~~~&&\forall \kappa \in \mathcal{P}_{r-2}(e_F),~\forall F \in \Delta_2(T),\\
            \label{dofL5}
           & \int_F (\btau \cdot \bn_F) \kappa \, dA ~~~~~~&&\forall \kappa \in \mathcal{R}_{r-1}(F^{ct}),~\forall F \in \Delta_2(T),\\
            \label{dofL6}
           & \int_F (\mathbf{curl}_F\btau_F )\kappa \, dA ~~~~&&\forall \kappa \in \mathring{\pol}_{r-2}(F^{ct}),~\forall F \in \Delta_2(T),\\
            \label{dofL7}
           & \int_F\btau_F \cdot \bkappa \, dA ~~~~&&\forall \bkappa \in \mathbf{grad}_F\,\mathring{S}_r(F^{ct}),~\forall F \in \Delta_2(T),\\
            \label{dofL8}
           & \int_T \bcurl\btau \cdot {\bm \kappa} \, dx~~~~~&&\forall {\bm \kappa} \in \bcurl \mathring{\bpol}_{r-1}^c(T^{wf}),\\ 
            \label{dofL9}
           & \int_T\btau \cdot \bkappa \, dx~~~~~&&\forall \bkappa \in \mathbf{grad}\, \mathring{S}_r(T^{wf}).
       \end{alignat}
\end{subequations}
 \item A function $\bp \in \bN_{r-2}(T^{wf})$  is uniquely defined by the following conditions:
  \begin{subequations} \label{equ:dofq}
      \begin{alignat}{2}
      \label{dofq1}
          & \int_e \llbracket \bp \cdot \bt_e \rrbracket_e \kappa\, ds ~~~~~~&&\forall \kappa \in \mathcal{P}_{r-3}(e),~\forall e \in \Delta_1^{I}(F^{ct})\backslash \{e_F\},~\forall F \in \Delta_2(T),\\
          \label{dofq2}
          & \int_{e_F} \llbracket \bp \cdot \bt_{e_F} \rrbracket_{e_F} \kappa\, ds ~~~~~~&&\forall \kappa \in \mathcal{P}_{r-2}(e_F),~\forall F \in \Delta_2(T),\\
          \label{dofq3}
          & \int_F (\bp \cdot \bn_F) \kappa\, dA ~~~~~~&&\forall \kappa \in \pol_{r-2}(F^{ct}),~\forall F \in \Delta_2(T),\\
          \label{dofq4}
          & \int_T (\Div \bp) \kappa\, dx ~~~~~&&\forall \kappa \in \mathring{\pol}_{r-3}(T^{wf}),\\
          \label{dofq5}
          & \int_T \bp \cdot \bkappa\, dx~~~~~&&\forall \bkappa \in \bcurl\mathring{\bpol}^c_{r-1}(T^{wf}).
      \end{alignat}
  \end{subequations}
 \item A function $v \in \pol_{r-3}(T^{wf})$ is uniquely defined by the following conditions:
  \begin{subequations} \label{equ:dofW}
      \begin{alignat}{1}
      \label{dofW1}
          & \int_T v\, dx\\
          \label{dofW2}
          & \int_T v \cdot \kappa \, dx~~~~~\forall \kappa \in \mathring{\pol}_{r-3}(T^{wf}).
      \end{alignat}
  \end{subequations}
  \end{enumerate}
 \end{lemma}

 \subsection{Equivalence norms assumptions} 
 We will use the above dofs to build Fortin-like projections. 
 Unfortunately, the dofs are not preserved with a Piola transform, as is the case for the N\'ed\'elec elements. Rather, our strategy to prove estimates of the projections is to map a tetrahedra to a unit size tetrahedra via dilation. To this end, we define scaled tetrahedra. 
 \begin{definition} \label{def:ScaledT}
For a tetrahedron $T\in \calT_h$, define its dilation and induced Worsey-Farin split
 \begin{equation}\label{eqn:ScaledT}
 \hat T = \frac1{h_T} T:=\{\frac{x}{h_T}:\ x\in T\},\qquad \hat T^{wf} = \{\frac1{h_T}K:\ K\in T^{wf}\}.
 \end{equation}
\end{definition}
Note  the scaled tetrahedra $\hat T$ are of unit size, and $\hat T^{wf}$ inherit the shape-regular properties of $T^{wf}$.
 From now on, unless otherwise specified, we denote $\hat \bv(\hat x) = \bv(h_T \hat x)$ \MJN{($\hat x\in \hat T$)} where 
 $\bv$ is either a scalar function or vector-valued function.

As a first step, we note that Lemma \ref{lem:LLDOFs} and equivalence of norms in finite dimensional spaces immediately give the following lemma. 
 \begin{lemma} \label{lemmaL}
 Consider $T \in \mathcal{T}_h$. 
 \begin{enumerate}
\item  Let $G^L_i$, $i=1,...,m_L$ be the functional given by dofs \eqref{equ:dofL} \MJN{on $\hat T$}, where $m_L$ is the dimension of $\bpol^c_{r-1}(\hat T^{wf})$. Then, there exists a constant $\beta({\hat T^{wf}})>0$ such that
\begin{equation*}
\|\hat \bv\|_{L^2(\hat{T})} \le \beta({\hat T^{wf}}) \sum_{i=1}^{m_L} |G^L_i( \hat\bv)|, \quad \forall  \hat \bv \in \bpol^c_{r-1}(\hat T^{wf}). 
\end{equation*}
\item Let $G^N_i$, $i=1,...,m_N$ be the functional given by dofs \SG{\eqref{equ:dofq}} \MJN{on $\hat T$}, where $m_N$ is the dimension of $\bN_{r-1}(\hat T^{wf})$. Then, there exists a constant $\theta({\hat T^{wf}})>0$ such that
\begin{equation*}
\| \widehat \bv\|_{L^2(\hat T)} \le \theta({\hat T^{wf}}) \sum_{i=1}^{m_L} |G^N_i( \widehat \bv)|, \quad \forall  \widehat \bv \in \bN_{r-1}(\hat T^{wf}). 
\end{equation*}
\end{enumerate}
\end{lemma}
 
In the following, we make an assumption concerning the uniform bound of the constants appearing in Lemma \ref{lemmaL}.  
\begin{assumption} \label{lem:equLN}
 There exists constants $\theta, \beta>0$ that depend only on the shape regularity of $\{\mathcal{T}_h\}$ such that 
 \begin{alignat*}{2}
\theta(\hat T^{wf}) \le & \theta, \quad && \forall T \in \mathcal{T}_h,\\
\beta(\hat T^{wf}) \le & \beta, \quad && \forall T  \in \mathcal{T}_h. 
 \end{alignat*}
\end{assumption}

\begin{remark}
One common approach (see for example \cite{walkington2014c}) to verify this assumption is to argue that the constants $\theta(\hat T^{wf})$ and $\beta(\hat T^{wf})$ are continuous functions of the vertices of $\hat T^{wf}$. If so, since the vertices live on a compact set, such a result implies the uniform bounds stated in Assumption \ref{lem:equLN}.
\end{remark}


\subsection{The modified Scott-Zhang interpolants}
The Fortin-like projections utilize Scott-Zhang-type 
interpolants onto piecewise linear polynomials
that preserve zero tangential or normal boundary conditions. The proof of the following lemma can be found in the appendix. To state the lemma we denote the patch around $T$: 
 \begin{equation*}
 \omega(T)=\bigcup\limits_{\substack{T'\in\mathcal{T}_h \\ \Bar{T}\cap\Bar{T'} \neq \emptyset }}T'.
 \end{equation*}
 \begin{lemma} \label{lem: SZ}
 Let $0<\delta \leq \frac{1}{2}$. There exists a projection $\boldsymbol{I}_h^{curl}: \bH^{1/2+\delta}(\Omega)  \longrightarrow \boldsymbol{\mathcal{P}}_1^c(\mathcal{T}_h)$
 with the following bounds:
 \begin{equation}\label{SZest}
     h_T^{-1/2-\delta}\|\btau-\boldsymbol{I}_h^{curl}\btau\|_{L^2(T)}+\|\boldsymbol{I}_h^{curl}\btau\|_{H^{1/2+\delta}(T)} \leq C\|\btau\|_{H^{1/2+\delta}(\omega(T))}~~~~~~~~\forall\btau \in \bH^{1/2+\delta}(\Omega)
 \end{equation}
for all $T\in \calT_h$.
 Moreover,  if $\btau \in \bH^{1/2+\delta}(\Omega) \cap \bH_0(\mathbf{curl},\Omega)$ then $\boldsymbol{I}_h^{curl}\btau \in \boldsymbol{\mathcal{P}}_1^c(\mathcal{T}_h) \cap \bH_0(\mathbf{curl},\Omega)$. 
 
 There also exists a projection $\boldsymbol{I}_h^{div}: \bH^{1/2+\delta}(\Omega) \longrightarrow \boldsymbol{\mathcal{P}}_1^c(\mathcal{T}_h)$   with the following bounds:
 \begin{equation}\label{SZestd}
     h_T^{-1/2-\delta}\|\btau-\boldsymbol{I}_h^{div}\btau\|_{L^2(T)}+\|\boldsymbol{I}_h^{div}\btau\|_{H^{1/2+\delta}(T)} \leq C\|\btau\|_{H^{1/2+\delta}(\omega(T))}~~~~~~~~\forall\btau \in  \bH^{1/2+\delta}(\Omega)
 \end{equation}
for all $T\in \calT_h$. Moreover,  if $\btau \in \bH^{1/2+\delta}(\Omega) \cap \bH_0(\Div,\Omega)$  then  $\boldsymbol{I}_h^{div}\btau \in \boldsymbol{\mathcal{P}}_1^c(\mathcal{T}_h) \cap \bH_0(\Div,\Omega)$. 
 \end{lemma}

 \section{Construction of Two Fortin-like Projections}\label{sec-Fortin}
In this section we define projections that conform to the framework
given in Section \ref{sec-framework}. Let 
\[ \bH^D(\bcurl,\Omega) := \{ \bv \in \bL^2(\Omega): \bcurl \bv \in \NN_{r-2}\} \subset \bH(\bcurl,\Omega). 
\]

\subsection{Definition of Fortin-like operators}
The following operators are defined
through the use of Lemma \ref{lem:LLDOFs}. 
We separate those degrees of freedom into two parts with one part 
affecting commuting properties and the second one not.  
 \begin{definition} \label{def:PiL}
 Define the operator $\bPi_{L}:\bH^D(\bcurl,\Omega) \cap \bH^{1/2+\delta}(\Omega) \rightarrow \LL_{r-1}$ such that on each $T \in \calT_h$,
{\small  \begin{subequations} \label{equ:PiL}
     \begin{alignat}{2}
     \label{PiL1}
          & \int_e (\bPi_{L} \bv \cdot \bt_e) \kappa \, ds= \int_e  (\bv \cdot \bt_e) \kappa \,ds\quad&& \forall \kappa \in \mathcal{P}_{r-3}(e),~\forall e \in \Delta_1(T),\\
          \label{PiL2}
         & \int_e \llbracket \bcurl\bPi_{L} \bv \cdot \bt_e \rrbracket_e \kappa \, ds = \int_e \llbracket \bcurl \bv \cdot \bt_e \rrbracket_e \kappa \,ds \ \ \ && \forall \kappa \in \mathcal{P}_{r-3}(e),~\forall e \in \Delta_1^{I}(F^{ct})\backslash \{e_F\}, \\
         & &&\ \ \forall F \in \Delta_2(T), \nonumber\\
          \label{PiL3}
          & \int_{e_F} \llbracket \bcurl\bPi_{L} \bv \cdot \bt_{e_F} \rrbracket_{e_F} \kappa \, ds =\int_{e_F} \llbracket \bcurl\bv \cdot \bt_{e_F} \rrbracket_{e_F} \kappa \, ds \ \ \ &&\forall \kappa \in \mathcal{P}_{r-2}(e_F),\ \forall F \in \Delta_2(T),\\
          \label{PiL4}
          & \int_F (\mathbf{curl}_F\,(\bPi_{L}\bv)_F )\kappa \, dA =\int_F (\mathbf{curl}_F \bv_F )\kappa \, dA ~~~~&&\forall \kappa \in \mathring{\pol}_{r-2}(F^{ct}),~\forall F \in \Delta_2(T),\\
          \label{PiL5}
           & \int_T \bcurl \bPi_{L}\bv \cdot {\bm \kappa} \, dx=\int_T \bcurl \bv \cdot {\bm \kappa} \, dx~~~~~&&\forall \kappa \in \bcurl\mathring{\bpol}_{r-1}^c(T^{wf}), \\
           \label{PiL6}
          & (\bPi_{L} \bv)(a)=0 ~~~~~~&&\forall a \in \Delta_0(T),\\
          \label{PiL7}
          & \int_e (\bPi_{L} \bv \times \bt_e) \cdot \bkappa \, ds = 0 ~~~~~~ 
          &&\forall \bkappa \in \bpol_{r-3}(e),~\forall e \in \Delta_1(T), \\
          \label{PiL8}
          & \int_F (\bPi_{L} \bv \cdot \bn_F) \kappa \, dA = 0 ~~~~~~&&\forall \kappa \in \mathcal{R}_{r-1}(F^{ct}),~\forall F \in \Delta_2(T),\\
          \label{PiL9}
           & \int_F \bPi_{L}\bv_F \cdot \bkappa \, dA =0  ~~~~&&\forall \bkappa \in \mathbf{grad}_F\, \mathring{S}_r(F^{ct}),~\forall F \in \Delta_2(T),\\
           \label{PiL10}
           & \int_T \bPi_{L}\bv \cdot \bkappa \, dx=0~~~~~&&\forall \bkappa \in \mathbf{grad}\, \mathring{S}_r(T^{wf}).
     \end{alignat}
 \end{subequations}}
 \end{definition}
 \begin{definition} \label{def:PiN}
 The operator $\bPi_N: \bH^{1/2+\delta}(\Omega) \cap \bH(\Div,\Omega) \rightarrow \NN_{r-2}$ is defined such that, on each $T \in \calT_h$,
{\small  \begin{subequations} \label{equ:PiN}
      \begin{alignat}{2}
      \label{PiN1}
          & \int_e \llbracket \bPi_{N} \bp \cdot \bt_e \rrbracket_e \kappa \, ds=0 ~~~~&&\forall \kappa \in \mathcal{P}_{r-3}(e),~\forall e \in \Delta_1^{I}(F^{ct})\backslash \{e_F\},~\forall F \in \Delta_2(T),\\
          \label{PiN2}
          & \int_{e_F} \llbracket \bPi_{N} \bp \cdot \bt_{e_F} \rrbracket_{e_F} \kappa \,ds=0 ~~~~&&\forall \kappa \in \mathcal{P}_{r-2}(e_F),~\forall F \in \Delta_2(T),\\
          \label{PiN3}
          & \int_F (\bPi_{N}\bp \cdot \bn_F) \kappa \, dA = \int_F (\bp \cdot \bn_F) \kappa \, dA ~~~~&&\forall \kappa \in \pol_{r-2}(F^{ct}),~\forall F \in \Delta_2(T),\\
          \label{PiN4}
          & \int_T (\Div \bPi_{N}\bp) \kappa \, dx=\int_T (\Div \bp) \kappa \, dx ~~~&&\forall \kappa \in \mathring{\pol}_{r-3}(T^{wf}),\\
          \label{PiN5}
          & \int_T \bPi_{N}\bp \cdot {\bm \kappa} \, dx=\int_T \bp \cdot \bkappa \, dx~~~&&\forall {\bm \kappa} \in \bcurl\mathring{\bpol}_{r-1}^c(T^{wf}).
      \end{alignat}
  \end{subequations}}
 \end{definition}
 \begin{definition}\label{def:PiW}
 The operator $\Pi_W:L^2(\Omega)\to \WW_{r-3}$ is defined such that, on each
 $T\in \calT_h$,
  \begin{subequations}\label{equ:PiW}
  \begin{alignat}{1}
  \label{PiW1}
          & \int_T \Pi_W v\, dx = \int_T v \, dx, \\
           \label{PiW2}
          & \int_T (\Pi_W v)  \kappa \, dx = \int_T v \kappa \, dx~~~~~\forall \kappa \in \mathring{\pol}_{r-3}(T^{wf}).
      \end{alignat}
  \end{subequations} 
 \end{definition}
We modify the operator $\bPi_{L}$ to obtain an Fortin-like projection that inherits the commuting
properties of $\bPi_L$ and approximation properties of the Scott-Zhang interpolant $\bI_h^{curl}$.
 \begin{definition}\label{def:PiV}
 The operator $\bPi_{V}:\bH^D(\bcurl,\Omega) \cap \bH^{1/2+\delta}(\Omega)  \rightarrow \LL_{r-1}$ is defined as
 \[
 \bPi_{V}=\bI_h^{curl}+\bPi_{L}({\bm 1}- \bI_h^{curl}),
 \]
 where ${\bm 1}$ is the identity operator and $\bPi_{L}$ is defined in Definition \ref{def:PiL}.
 \end{definition}
The next two theorems state the commuting properties of $\bPi_{V}$ and  $\bPi_N$ and their
approximation properties in the $L^2$-norm.
The proof of these theorems 
are postponed to Section \ref{sec-theProofsTh}.
\begin{thm} \label{thm:PiV}
The operator $\bPi_{V}:\bH^D(\bcurl,\Omega) \cap \bH^{1/2+\delta}(\Omega) \rightarrow \LL_{r-1}$ defined in Definition \ref{def:PiV} satisfies 
 \begin{alignat}{1}
 \label{eqn:PiVCommute}
 \mathbf{curl}\,\bPi_{V} \btau &= \mathbf{curl} \,\btau, \quad \forall \btau \in \bH^D(\bcurl,\Omega) \cap \bH^{1/2+\delta}(\Omega).
 \end{alignat}
 Moreover, under Assumption \ref{lem:equLN}, the following bound holds for any $\btau \in \bH^D(\bcurl,\Omega) \cap \bH^{1/2+\delta}(\Omega)$:
 \begin{equation}\label{boundV}
  \|\bPi_{V} \btau - \btau\|_{L^2(\Omega)}\le C\big(h^{1/2+\delta} \|\btau\|_{H^{1/2+\delta}(\Omega)} + h \|{\bf curl} \btau \|_{L^2(\Omega)}\big).
\end{equation}
Finally, if $\btau \in  \bH^D(\bcurl,\Omega) \cap \bH^{1/2+\delta}(\Omega) \cap \bH_0(\mathbf{curl},\Omega)$ then $ \bPi_{V} \btau \in \LL_{r-1}^t$. 
 \end{thm}
 We also state the analogous result for $\bPi_{N}$. 
\begin{thm} \label{thm:Nest}
 The operator $\bPi_{N}: \bH^{1/2+\delta}(\Omega) \cap \bH(\Div,\Omega) \rightarrow \NN_{r-2}$ defined in Definition \ref{def:PiL} satisfies
 \begin{equation} \label{eqn:PiNCommute}
 \Div \bPi_{N} \bp = \Pi_W \Div \bp \qquad \forall \bp \in \bH^{1/2+\delta}(\Omega) \cap \bH({\rm div},\Omega).
 \end{equation}
 Moreover, under Assumption \ref{lem:equLN}, the following bound holds for any $\bp \in \bH^{1/2+\delta}(\Omega) \cap \bH({\rm div},\Omega)$:
 \begin{equation}\label{boundN}
     ||\bPi_{N} \bp-\bp||_{L^2(\Omega)} \leq C\big(h^{1/2+\delta}||\bp||_{\bH^{1/2+\delta}(\Omega)} + h\|\Div \bp \|_{L^2(\Omega)}\big). 
 \end{equation}
 Finally, if $\bp \in  \bH^{1/2+\delta}(\Omega) \cap \bH_0({\rm div},\Omega)$ then $ \bPi_{N} \bp \in \NN_{r-2}^n$. 
 \end{thm}

\subsection{Bounds after Dilation}
\subsubsection{Some inequalities on scaled tetrahedra}
We first give several inequalities on the scaled tetrahedra $\hat T$ from Definition \ref{def:ScaledT}. The proofs of the
next three results are shown in the appendix.

\begin{prop} \label{prop:dialT}
Under Definition \ref{def:ScaledT}, we have the following results:
\begin{itemize}
\item[(i)]  There holds ${\rm diam}(\hat T) =1 \le c_0 \rho_{\hat T}$, 
where $c_0$ is the shape-regularity constant given in Definition \ref{def:calKDef}.
\item[(ii)] $\widehat{\Div} \hat \bv = h_T \Div \bv$  and $\widehat{\bcurl} \hat \bv = h_T \bcurl \bv$. 
\item[(iii)] There holds for $0 \le s\le 1$  
 \begin{equation*}
      c_0^{-3/2-s}h_T^{-3/2+s}|\bv|_{H^{s}(T)}\le  |\hat \bv|_{H^{s}(\hat T)} \le h_T^{-3/2+s}|\bv|_{H^{s}(T)}\qquad \forall \bv\in \bH^s(T).
 \end{equation*} 
 \end{itemize}
\end{prop}

The next result establishes a trace inequality on $\hat T$.
\begin{lemma} \label{lem:trace}
For any $\hat \bv  \in  \bH^{1/2+\delta}(\hat T)\ (\delta\in (0,\frac12])$, we have 
$\hat \bv|_{\p \hat T}\in \bL^p(\p \hat T)$ for $2 \le p < \frac2{1-\delta}$.  
In particular,
\[
\|\hat \bv\|_{L^p(\p \hat T)}\le C\|\hat \bv\|_{H^{1/2+\delta}(\hat T)}\qquad \forall \hat \bv\in \bH^{1/2+\delta}(\hat T),
\]
where $C$ is a uniform constant for all $\hat T$. 
\end{lemma}

The following lemma gives an inverse trace operator with estimate.
Its proof is given in \cite[Lemma 4.7]{VP1998};
however, for completeness, we provide the proof with additional details in the appendix. 
\begin{lemma} \label{lem:lift}
Let $\hat e\in \Delta_1(\hat T)$, and let $\hat F\in \Delta_2(\hat T)$ 
be a face that has $\hat e$ as an edge. 
Then, there exists an extension operator 
$E: \pol_{r-3}(\hat e) \rightarrow W^{1,q}(\hat T)$ with $1<q<2$ such that $(E \hat \kappa)|_{\hat e} = \hat \kappa$, $E \hat \kappa|_{\partial \hat F \backslash \hat e}=0$ and $E \hat \kappa|_{\partial \hat T \backslash \hat F}=0$. Moreover, the following estimates hold:
\begin{subequations}
\begin{alignat}{1}
    \| E\hat \kappa \|_{W^{1,q}(\hat F)} &\le C_1 \|\hat \kappa\|_{W^{1-1/q,q}(\hat e)}, \label{lem:lift1}\\
 \|E\hat \kappa\|_{W^{1,q}(\hat T)} &\le C_2 \|\hat \kappa\|_{W^{1-1/q,q}(\hat e)}, \label{lem:lift2}
\end{alignat}
\end{subequations}
where $C_1, C_2 >0$ are two constants uniform for all $\hat T$.
\end{lemma}

\subsubsection{The estimates}
 The following two lemmas not only show that the operators $\bPi_L$ and $\bPi_N$ are well-defined on their respective domains but also 
 give local estimates on the tetrahedra after dilating.

 \begin{lemma} \label{lem: liftL}
 Let $\hat{\bPi}_L$ be operator given in Definition \ref{def:PiL} defined on $\hat{T}$ \SG{\eqref{eqn:ScaledT}}. Under Assumption \ref{lem:equLN}, there holds
    \begin{equation}\label{eqn:hatPiLEstimate}
      \|\hat{\bPi}_L \hat \bv\|_{L^2(\hat{T})} \leq C(\|\widehat \bcurl \hat \bv\|^2_{L^2(\hat{T})}+\|\hat \bv\|_{H^{1/2+\delta}(\hat{T})}),
  \end{equation}
  for all $\hat \bv\in \bH^{1/2+\delta}(\hat T)$ with $\widehat \bcurl \hat \bv\in \bpol_{r-2}(\hat T^{wf})\cap \bH(\widehat{\rm div},\hat T)$.
  \end{lemma}

\begin{proof}
We bound the corresponding non-zero functionals appearing on the right-hand side of Definition \ref{def:PiL}.
We start with the the following estimates,
which follow from H\"older's inequality and inverse estimates
\SG{\eqref{inv2}} that hold since $\widehat \bcurl \hat \bv$ is  a piecewise polynomial. 
\begin{alignat*}{2}
      |\int_{\hat e} \llbracket \widehat\bcurl \hat \bv \cdot \bt \rrbracket_{\hat e} \hat \kappa\,ds| 
      & \leq C\|\widehat\bcurl \hat \bv\|_{L^2( \hat{T})} \|\hat \kappa\|_{L^{2}(\hat e)} \quad &&\forall \hat \kappa \in L^{2}(\hat e),\ \forall  \hat e\in \Delta_1(\hat T^{wf}), \hat e \subset \hat F\\ 
   |\int_{\hat{F}} (\widehat{\mathbf{curl}}_{\hat{F}} \hat \bv_{\hat{F}} )  \hat \kappa\,dA| 
   & \leq C\|\widehat \bcurl \hat \bv\|_{L^2( \hat{T})} \|\hat \kappa\|_{L^{2}(\hat{F})}\quad &&
       \forall \hat \kappa \in L^{2}(\hat F),\ \forall \hat F\in \Delta_2(\hat T),\\
    |\int_{\hat{T}} \widehat \bcurl \hat \bv \cdot \hat \bkappa\, dx| &\leq \|\widehat \bcurl \hat \bv\|_{L^2(\hat{T})} \|\hat \bkappa\|_{L^{2}(\hat{T})}\quad  &&\forall \hat \bkappa \in \bL^{2}(\hat T).
       \end{alignat*}

We now bound the remaining functionals coming from the right-hand side of \eqref{PiL1}
by adopting the technique developed in the proof of \cite[Lemma 4.7]{VP1998}. 
Let $\hat e \in \Delta_1(\hat T)$ and let $\hat F \in \Delta_2(\hat T)$ have $\hat e$ as an edge.  
We choose $p$ such that $2 < p < \frac2{1-\delta}$ in order to apply Lemma \ref{lem:trace}. Let $\hat \kappa \in \pol_{r-3}(\hat e)$,
and let $E\hat \kappa\in W^{1,p'}(\hat T)$ be as in Lemma \ref{lem:lift} (with $q=p'<2$, the H\"older conjugate of $p$).
Integration by parts and using $E \hat \kappa|_{\hat e} =\hat \kappa$ and $E \hat \kappa|_{\partial \hat{F}\backslash \hat e} =0$ gives
 \begin{equation*}
      \int_{\hat e}  (\hat \bv \cdot \hat \bt) \hat \kappa\, ds  =\int_{\hat F}  (\widehat \bcurl \hat \bv) \cdot \hat \bn  ~(E \hat \kappa)\,d\hat A
      +\int_{\hat F} (\hat \bv\times \hat \bn) \cdot  (\widehat{\mathbf{grad}}\,E\hat \kappa)\,d\hat A.
\end{equation*}
An additional integration by parts, using $E \hat \kappa|_{\partial \hat T \backslash \hat F}=0$, H\"olders inequality, estimate \eqref{lem:lift2}, \SG{an inverse estimate \eqref{inv1} \SG{on $\hat T^{wf}$} and the shape regularity of $\hat T^{wf}$} 
gives  
\begin{alignat*}{1}
|\int_{\hat F}  (\widehat \bcurl \hat \bv) \cdot \hat \bn  ~(E \hat \kappa)\,d\hat A|
= & |\int_{\hat T}  \widehat \bcurl \hat \bv \cdot \widehat{\mathbf{grad}}\, E \hat \kappa\, d\hat x| \\
\le & C \|\widehat \bcurl \hat \bv\|_{L^p(\hat{T})} \| \hat \kappa\|_{W^{1-1/p',p'}(\hat e)} \\
\le & C \|\widehat \bcurl \hat \bv\|_{L^2(\hat{T})} \| \hat \kappa\|_{W^{1-1/p',p'}(\hat e)}.
\end{alignat*}

Using H\"older's inequality, Lemma \ref{lem:trace}, and  \eqref{lem:lift2} we obtain
\begin{alignat*}{1}
|\int_{\hat F} (\hat \bv\times \hat \bn) \cdot  (\widehat{\mathbf{grad}}~E\hat \kappa)~d\hat A | \le  \|\hat \bv\times \hat \bn\|_{L^p(\hat F)}\|  E\hat \kappa\|_{W^{1,p'}(\hat F)}
\le  \SG{C} \|\hat \bv\|_{H^{1/2+\delta}(\hat T)}\| \hat \kappa\|_{W^{1-1/p',p'}(\hat e)}.
\end{alignat*}

Combining the above estimates and using Definition \ref{def:PiL} with Assumption \ref{lem:equLN} 
yields the result \eqref{eqn:hatPiLEstimate}.
\end{proof}

We now derive a similar estimate for $\hat \bPi_N$. 
\begin{lemma} \label{lem: liftN}
 Let $\hat \bPi_N$ be the operator given
 in Definition \ref{def:PiN} defined on $\hat T$ \eqref{eqn:ScaledT}.
Under Assumption \ref{lem:equLN}, there holds with a uniform constant $C$ for all $\hat T$,
  \begin{equation}
      \|\hat{\bPi}_N\hat \bp \|_{L^2(\hat{T})} \leq C(\|\widehat\Div  \hat \bp\|_{L^2(\hat{T})}+ \|\hat \bp\|_{H^{1/2+\delta}(\hat{T})}) \quad \forall \hat\bp \in  \bH^{1/2+\delta}(\hat T) \cap \bH(\widehat{\rm div},\hat T). 
  \end{equation}
  \end{lemma}
  \begin{proof}
 Using the estimate in Lemma \ref{lem:trace}, \MJN{the Cauchy-Schwarz inequality},
 along with $\widehat \Div \hat \bp\in L^2(\hat T)$, we obtain
\begin{alignat*}{2}
\Big| \int_{\hat F} (\hat \bp\cdot \hat \bn_{\hat F}) \hat \kappa\, d\hat A\Big|
&\le C \|\hat \bp\|_{H^{1/2+\delta}(\hat T)}\|\hat \kappa\|_{L^2(\hat F)}\quad &&\forall \hat \kappa\in L^2(\hat F),\ \forall \hat F\in \Delta_2(\hat T),\\
\Big|\int_{\hat T} (\widehat \Div \hat \bp)\hat \kappa\, d\hat x\Big|
&\le \|\widehat \Div \hat \bp\|_{L^2(\hat T)} \|\hat \kappa\|_{L^2(\hat T)}\quad &&\forall \hat \kappa\in L^2(\hat T),\\
\Big|\int_{\hat T} \hat \bp\cdot \hat \bkappa\, d\hat x\Big| 
&\le \|\hat \bp\|_{H^{1/2+\delta}(\hat T)} \|\hat \bkappa\|_{L^2(\hat T)}\quad &&\forall \hat \bkappa\in \bL^2(\hat T).
\end{alignat*}

These estimates, combined with Lemma \ref{lem:LLDOFs}, Definition \ref{def:PiN}, and Assumption \ref{lem:equLN} yield
  \begin{equation*}
   \|\hat{\bPi}_N \hat \bp \|_{L^2(\hat{T})} \leq C(\|\widehat \Div \hat \bp\|_{L^2(\hat{T})}+ \|\hat \bp\|_{H^{1/2+\delta}(\hat{T})}), 
  \end{equation*}
where $C$ is a uniform constant for all $\hat T$.
  \end{proof}

\subsection{Proofs of Theorems \ref{thm:PiV}--\ref{thm:Nest}}\label{sec-theProofsTh}
In order to prove these theorems, we transfer the results for $\hat T$ back to $T$. We start with the proof of Theorem \ref{thm:Nest}.
\begin{proof}[Proof of Theorem \ref{thm:Nest}.]
 We first prove  \eqref{eqn:PiNCommute}. Let $\bp \in \bH^{1/2+\delta}(\Omega) \cap \bH({\rm div},\Omega)$ and  set $\rho = (\Div \bPi_{N} \bp - \Pi_W {\rm div} \bp) \in \WW_{r-3}$.
 First, by \eqref{PiN3},  \eqref{PiW1} and Stokes theorem, we have on each $T\in \calT_h$,
   \begin{equation*}
       \int_T \rho\, dx = \int_T {\rm div}(\bPi_{N} \bp-\bp)\, dx = \int_{\partial T} (\bPi_{N} \bp-\bp) \cdot \bn \, dA =0,
   \end{equation*}
   where we use that the constant functions are in $\pol_{r-2}(F^{ct})$ . Next, for any $\kappa \in \mathring{\pol}_{r-3}(T^{wf})$, we use  \eqref{PiN4} and  \eqref{PiW2} to obtain  
  \begin{equation*}
 \int_T     \rho \kappa\, dx = \int_T \Div (\bPi_{N} \bp-\bp) \kappa \, dx=0.
  \end{equation*}
  Thus $\rho=0$ by \eqref{equ:dofW}, and so \eqref{eqn:PiNCommute} holds.

Next we prove the bound \eqref{boundN}. For any $T \in \calT_h$, with $\hat T$ defined in \eqref{eqn:ScaledT} and $\hat \bp (\hat x) = \bp(h_T \hat x)$, $\forall \hat x \in \hat T$, it is easy to check that $\widehat{\bPi_N \bp} = \hat \bPi_N \hat \bp$ by Definition \ref{def:PiN} and Lemma \ref{lem:LLDOFs}. With Lemma \ref{lem: liftN} and Proposition \ref{prop:dialT}, we have:
\begin{equation} \label{equ:PiNest}
\begin{split}
   \|\bPi_N \bp\|_{L^2(T)}^2 
   & = h_T^{3} \|\widehat{{\bPi}_N \bp}\|^2_{L^2(\hat{T})} 
   = h_T^{3} \|\hat{\bPi}_N \hat{\bp}\|^2_{L^2(\hat{T})} \\
   & \leq Ch_T^{3}(\|\widehat \Div \hat \bp\|^2_{L^2(\hat{T})}+ \|\hat \bp\|^2_{H^{1/2+\delta}(\hat{T})})
   \\
   & = Ch_T^{3}(\|\widehat \Div \hat \bp\|^2_{L^2(\hat T)} + |\hat \bp|^2_{H^{1/2+\delta}(\hat T)}+\|\hat \bp\|^2_{L^2(\hat T)}) \\
   & \leq C(h_T^2\|\Div \bp\|^2_{L^2(T)} + h_T^{1+2\delta}| \bp|^2_{H^{1/2+\delta}(T)}+\|\bp\|^2_{L^2(T)}),
\end{split}
\end{equation}
where the constant $C$ is independent of $h_T$.

Then by Lemma \ref{lem: SZ}, \eqref{equ:PiNest} and the inverse estimate \eqref{inv1} , we have
  \begin{equation*}
  \begin{split}
       \|\bPi_{N}\bp-\bp\|^2_{L^2(T)} & \leq 2\|\bPi_{N}\bp-\boldsymbol{I}_h^{div} \bp\|^2_{L^2(T)} +2\|\boldsymbol{I}_h^{div} \bp-\bp\|^2_{L^2(T)} \\
       & = 2\|\bPi_{N}(\boldsymbol{I}_h^{div} \bp-\bp)\|^2_{L^2(T)}+2\|\boldsymbol{I}_h^{div} \bp-\bp\|^2_{L^2(T)} \\
       & \leq C\big(h_T^2\|\Div \boldsymbol{I}_h^{div} \bp\|^2_{L^2(T)}+h_T^2\|\Div \bp\|^2_{L^2(T)} \\
      & \quad \qquad +h_T^{1+2\delta}|\boldsymbol{I}_h^{div} \bp-\bp|^2_{H^{\frac{1}{2}+\delta}(T)}  +\|\boldsymbol{I}_h^{div} \bp-\bp\|^2_{L^2(T)}\big)\\
      & \leq C(h^2_T\|\nabla \boldsymbol{I}_h^{div} \bp\|^2_{L^2(T)}+h_T^2\|\Div \bp\|^2_{L^2(T)}+h_T^{1+2\delta}\|\bp\|^2_{H^{\frac{1}{2}+\delta}(\SG{\omega(T)})})\\
      & \leq C\big(h_T^{1+2\delta}\| \boldsymbol{I}_h^{div} \bp\|^2_{H^{\frac{1}{2}+\delta}(T)}+h_T^{1+2\delta}\|\bp\|^2_{H^{\frac{1}{2}+\delta}(\SG{\omega(T)})}+h_T^2\|\Div \bp\|^2_{L^2(T)}\big) \\
      & \leq C\big(h_T^{1+2\delta}\|\bp\|^2_{H^{\frac{1}{2}+\delta}(\SG{\omega(T)})}+h_T^2\|\Div \bp\|^2_{L^2(T)}\big).
  \end{split}
  \end{equation*}
Summing this result over all tetrahedra $T\in \calT_h$ gives \eqref{boundN}.

Finally, if $\bp \in \bH^{1/2+\delta}(\Omega) \cap \bH_0({\rm div},\Omega)$, then it easily follows from \eqref{PiN3} that $\bPi_{N} \bp \cdot \bn=0$ on $\partial \Omega$,
which implies $\bPi_{N} \bp \in \NN_{r-2}^n$. 
\end{proof}

Next we turn our attention
to Theorem \ref{thm:PiV}.
To this end, we first state
and prove an intermediate result
for the operator $\bPi_L$.
\begin{lemma} \label{lem:Lest}
Under Assumption \ref{lem:equLN}, the operator $\bPi_{L}: \bH^D(\bcurl,\Omega) \cap \bH^{1/2+\delta}(\Omega) \rightarrow \LL_{r-1}$ defined in Definition  \ref{def:PiL}  satisfies 
  \begin{equation}\label{eqn:PiLCommute}
 \mathbf{curl}\,\bPi_{L} \btau = \mathbf{curl}\,\btau\qquad \forall \btau\in \bH^D(\bcurl,\Omega) \cap \bH^{1/2+\delta}(\Omega).
 \end{equation}
 Moreover, the following bound holds for any $\btau \in \bH^D(\bcurl,\Omega) \cap \bH^{1/2+\delta}(\Omega)$,
 \begin{equation}\label{boundL}
 \|\bPi_{L} \btau\|_{L^2(\Omega)}\le C \big(\|\btau\|_{L^2(\Omega)}+ h^{1/2+\delta}|\btau|_{H^{1/2+\delta}(\Omega)}+ h \|{\bf curl} \btau\|_{L^2(\Omega)}\big).
 \end{equation}
Finally, if $\btau \in  \bH^D(\bcurl,\Omega) \cap \bH^{1/2+\delta}(\Omega) \cap \bH_0(\mathbf{curl},\Omega)$ then $ \bPi_{L} \btau \in \LL_{r-1}^t$. 
 \end{lemma}
%
\begin{proof}
We first prove \eqref{eqn:PiLCommute}. Set $\bm \rho = \mathbf{curl}\,\bPi_{L} \btau - \mathbf{curl}\,\btau \in \NN_{r-2}$.
 Then by directly using the definition of $\bPi_{L}$, 
 we see that $\bm \rho$ vanishes on the DOFs \eqref{dofq1}--\eqref{dofq2},\eqref{dofq4}--\eqref{dofq5}, and
 \[
 \int_F (\bm \rho\cdot \bn_F)\kappa\, dA =0\qquad \forall \kappa\in \mathring{\pol}_{r-2}(F^{ct}),
 \]
 where we used the identity ${\bf curl}_F \btau_F = {\bf curl}\, \btau\cdot \bn_F$.
 Next, by Stokes Theorem
 \begin{equation*}
 \begin{split}
     \int_F \bm \rho \cdot \bn_F\, dA & = \int_F \mathbf{curl}_F( \bPi_{L} \btau-\btau)_F\, dA \\  
     &=\sum_{e\in \Delta_1(F)} \int_e  ( \bPi_{L} \btau - \btau)\cdot \bt_e ds=0.
 \end{split}
 \end{equation*}
 Since the difference between ${\pol}_{r-2}(F^{ct})$ and  $\mathring{\pol}_{r-2}(F^{ct})$ is the space of the constant functions on $F$, 
we have
 \[
 \int_F (\bm \rho\cdot \bn_F)\kappa\, dA =0\qquad \forall \kappa\in {\pol}_{r-2}(F^{ct}),
 \]
 and so $\bm \rho$ vanishes on all the DOFs \eqref{equ:dofq}.  We thus conclude $\bm \rho \equiv 0$ by Lemma \ref{lem:LLDOFs}, and therefore \eqref{eqn:PiLCommute} is satisfied.

We now prove \eqref{boundL}. For any $T \in \calT_h$, with $\hat T$ defined in \eqref{eqn:ScaledT} and $\hat \btau (\hat x) = \btau(h_T \hat x)$, $\forall \hat x \in \hat T$, it is easy to check that $\widehat{\bPi_L \btau} = \hat \bPi_L \hat \btau$ by Definition \ref{def:PiL} and Lemma \ref{lem:LLDOFs}. With Lemma \ref{lem: liftL} and Proposition \ref{prop:dialT}, for any $T \in \calT_h$, we have:
\begin{equation} \label{equ:PiLest}
\begin{split}
   \|\bPi_L \btau\|^2_{L^2(T)} & = h_T^{3} \|\hat{\bPi}_L \hat{\btau}\|^2_{L^2(\hat{T})} \leq Ch_T^{3}(\|\widehat \bcurl \hat \btau \|^2_{L^2(\hat{T})}+\|\hat \btau\|^2_{H^{1/2+\delta}(\hat{T})}) \\
   & = Ch_T^{3}(\|\widehat \bcurl \hat \btau \|^2_{L^2(\hat{T})}+\|\hat \btau\|^2_{L^2(\hat{T})}+ | \hat \btau|^2_{H^{1/2+\delta}(\hat T)})\\
   & \leq C(h_T^2\|\bcurl \btau \|^2_{L^2(T)} + \| \btau \|^2_{L^2(T)}+h_T^{1+2\delta}| \btau |^2_{H^{1/2+\delta}(T)}),
\end{split}
\end{equation}
where the constant $C$ is independent of $h_T$. 
Then summing up all the tetrahedra gives the bound \eqref{boundL}.

Finally, we will show that if $\btau \in  \bH^D(\bcurl,\Omega) \cap \bH^{1/2+\delta}(\Omega) \cap \bH_0(\mathbf{curl},\Omega)$, then $ \bPi_{L} \btau \in \LL_{r-1}^t$.  Since $\bPi_{L} \btau \in \LL_{r-1}$, we only need to show $\bPi_{L} \btau \times \bn = 0$ on $\partial \Omega$. Note that  $\btau \times \bn_F=0$ and $\bcurl \btau \cdot \bn_F =0$ on $F$ for all 
$F\in \Delta_2(\mathcal{T}_h)$ with $F \subset \p \Omega$. Let $F$ be a boundary face and
let $e \in \Delta_1(F)$. Then for all $\kappa \in \mathcal{P}_{r-3}(e)$, recalling $E\kappa$ in Lemma \ref{lem:lift}, we use integration by parts to get
\begin{equation*}
     \begin{aligned}
          & \int_e (\bPi_{L} \btau \cdot \bt_e) \kappa \, ds= \int_e  (\btau \cdot \bt_e) \kappa ds = \int_F(\bcurl \btau)\cdot \bn_F \, E\kappa\, dA+\int_F(\btau \times \bn_F)\cdot \mathbf{grad}~E\kappa\, dA = 0.
     \end{aligned}
 \end{equation*}
 Therefore, because $\bPi_L \btau$ vanishes
at the vertices of $F$, 
$(\bPi_{L} \btau \cdot \bt_e)|_e = 0$ and hence $\mathbf{curl}_F~ (\bPi_{L}\btau)_F \in \mathring{\pol}_{r-2}(F^{ct})$. By \eqref{PiL4} we get $\mathbf{curl}_F~ (\bPi_{L}\btau)_F=0$. Using  the exactness on Clough-Tocher splits \cite[(3.3e)]{Ex2020}, we know $(\bPi_{L}\btau)_F \in \mathbf{grad}_F~\mathring{S}^0_r(F^{ct})$ which after applying
\eqref{PiL9} shows $(\bPi_L \btau)_F =0$ on $F$. 
 Since $F \subset \partial \Omega$ was arbitrary, we conclude that $\bPi_{L} \btau \times \bn =0$ on $\partial \Omega$ and thus $\bPi_{L} \btau \in \LL_{r-1}^t$.  
\end{proof}

We finish with the proof  of Theorem \ref{thm:PiV}.
\begin{proof}[Proof of Theorem \ref{thm:PiV}.]

The commuting property \eqref{eqn:PiVCommute} easily follows from \eqref{eqn:PiLCommute}. Indeed,
 \begin{align*}
 {\bf curl}\,\bPi_{V}\btau
 &={\bf curl}(\bI_h^{curl} \btau)+{\bf curl}(\bPi_{L}(\btau- \bI_h^{curl}\btau))\\
 &={\bf curl}(\bI_h^{curl} \btau)+{\bf curl}(\btau- \bI_h^{curl}\btau)\\
 &=\mathbf{curl}\,\btau.
 \end{align*}
Since  $\bI_h^{curl} \btau \in \LL_{r-1}^t$ if $\btau \in  \bH^D(\bcurl,\Omega) \cap \bH^{1/2+\delta}(\Omega) \cap \bH_0(\mathbf{curl},\Omega)$, with Lemma \ref{lem:Lest}, we have $\bPi_V \btau \in \LL_{r-1}^t$.

To prove \eqref{boundV} we use Lemma \ref{lem: SZ} and Lemma \ref{lem:Lest} to obtain
\begin{equation}
    \begin{split}
         \|\bPi_{V} \btau - \btau\|_{L^2(\Omega)} & \le \|\bI_h^{curl} \btau-\btau\|_{L^2(\Omega)} + \|\bPi_L(\bI_h^{curl} \btau-\btau)\|_{L^2(\Omega)} \\
    & \le C\big(\|\bI_h^{curl} \btau-\btau\|_{L^2(\Omega)}+ h^{1/2+\delta} |\bI_h^{curl} \btau-\btau|_{H^{1/2+\delta}(\Omega)} \\
    & \quad \quad \quad + h \|{\bf curl} (\bI_h^{curl} \btau) \|_{L^2(\Omega)}+ h \|{\bf curl} \btau \|_{L^2(\Omega)}\big) \\
    & \le C\big(h^{1/2+\delta} \|\btau\|_{H^{1/2+\delta}(\Omega)} + h \| \nabla \bI_h^{curl} \btau \|_{L^2(\Omega)}+ h \|{\bf curl} \btau \|_{L^2(\Omega)}\big) \\
    & \le C\big(h^{1/2+\delta} \|\btau\|_{H^{1/2+\delta}(\Omega)} + h \|{\bf curl} \btau \|_{L^2(\Omega)}\big),
    \end{split}
\end{equation}
where we used the inverse inequality \eqref{inv1}.

Finally, we will show that if $\btau \in  \bH^D(\bcurl,\Omega) \cap \bH^{1/2+\delta}(\Omega) \cap \bH_0(\mathbf{curl},\Omega)$,
then $ \bPi_{V} \btau \in \LL_{r-1}^t$. Similar to the proof for 
$\bPi_L$, we only need to check the boundary conditions. Since 
$\btau \times \bn=0$ on $\partial \Omega$, by Lemma \ref{lem: SZ}, 
$\bI_h^{curl} \btau \times \bn = 0$ on $\partial \Omega$ and by 
Lemma \ref{lem:Lest}, $\bPi_{L} \btau \times \bn =0$ on $\partial 
\Omega$. Therefore, $\bPi_{L} \bI_h^{curl} \btau \times \bn =0$ on 
$\partial \Omega$ and thus, 
\[
\bPi_V \btau \times \bn = \bI_h^{curl} \btau \times \bn+\bPi_{L} 
\btau \times \bn - \bPi_{L} \bI_h^{curl} \btau \times \bn = 0\quad 
\text{on }\p\Omega.
\]
This gives $\bPi_{V} \btau \in \LL_{r-1}^t$.
\end{proof}

\section{Application to the Maxwell eigenvalue problem}\label{apply-Max}
In this section, we apply the convergence theory established
in Section \ref{subsec-conv} and the properties of the two Fortin-like projections to show that quadratic (or higher) Lagrange finite element on Worsey-Farin meshes lead to convergent approximations of the Maxwell eigenvalue problem \eqref{equ:Ceig}.  
First, we require the following proposition.
\begin{prop} \label{prop:exact}
Recall the domain $\Ome$ is contractible. Then the complex \eqref{equ:exact} and \eqref{equ:bspace} are exact sequences. 
In particular,
\[
\bcurl \LL^t_{r-1} = \mathbf{\Psi}_{r-2}:=\mathbf{ker}(\NN^n_{r-2}, \Div).
\]
\end{prop}
This result follows from the Bogovskii 
operator 
in \cite[Theorem 4.9]{ON2010} and the projections $\bPi_L$, $\bPi_N$. We omit the details.

In Assumption \ref{VQ}, set $\VV_h = \LL^t_{r-1}$, $\QQ_h = \mathbf{\Psi}_{r-2}$ and $$\VV^t(\mathbf{\Psi}_{r-2}):=\{\btau \in \VV^t: \bcurl \tau \in \mathbf{\Psi}_{r-2}\}.$$
Then Theorems 
\ref{thm:PiV}--\ref{thm:Nest} 
 lead to the following results:
\begin{cor} \label{cor:assp}
The projection $\bPi_{V}:\VV^t(\mathbf{\Psi}_{r-2}) \to \LL_{r-1}^t $ satisfies
\begin{alignat*}{2}
    \bcurl\bPi_{V}\btau &= \bcurl\btau ~~~~~~~ &&\forall\btau \in \VV^t(\mathbf{\Psi}_{r-2}),\\
    \|\bPi_{V} \btau - \btau\|_{L^2(\Omega)}&\le C\big(h^{1/2+\delta} \|\btau\|_{H^{1/2+\delta}(\Omega)} + h \|{\bf curl} \btau \|_{L^2(\Omega)}\big)~~~~~~~ &&\forall\btau \in \VV^t(\mathbf{\Psi}_{r-2}).
\end{alignat*}
Furthermore,  the $L^2$-orthogonal projection $\bbP_{\Psi}:\bL^2(\Omega) \rightarrow \mathbf{\Psi}_{r-2}$ satisfies
 \begin{equation*}
     \|\bbP_{\Psi} \bp-\bp\|_{L^2(\Omega)} \leq Ch^{1/2+\delta}\|\bcurl \bp\|_{L^2(\Omega)}~~~~~\forall \bp \in \bH(\mathbf{curl}, \Omega) \cap \bH_0({\rm div}^0, \Omega).
\end{equation*}
\end{cor}
 \begin{proof}
 Let $\delta\in (0,\frac12]$ be the constant in the embedding
 result of Proposition \ref{emb}.
 The results for the operator $\bPi_{V}$ is then a direct consequence of Theorem \ref{thm:PiV} since $\VV^t(\mathbf{\Psi}_{r-2}) \subset \bH^D(\bcurl,\Omega) \cap \bH^{1/2+\delta}(\Omega)$. 
 
 Next, we prove the estimate for $\bbP_{\Psi}$. Since $\bH(\mathbf{curl}, \Omega) \cap \bH_0({\rm div}^0, \Omega) \subset \bH^{1/2+\delta}(\Omega) \cap \bH_0(\Div,\Omega)$ 
 , we use Theorem \ref{thm:Nest} to obtain $\Div \bPi_N \bp = \Pi_W \Div \bp = 0$. Thus, $\bPi_N \bp \in \mathbf{\Psi}_{r-2}$. Consequently, we have the estimate for $\bbP_{\Psi}$:
 \begin{equation*}
 \begin{split}
     \|\bbP_{\Psi} \bp-\bp\|_{L^2(\Omega)} & \leq  \|\bPi_{N}\bp-\bp\|_{L^2(\Omega)} \leq Ch^{1/2+\delta}\|\bp\|_{H^{1/2+\delta}(\Omega)} \\
     & \leq Ch^{1/2+\delta} (\|\bcurl\bp \|_{L^2(\Omega)}+\|\Div\bp \|_{L^2(\Omega)})=Ch^{1/2+\delta}\|\bcurl\bp\|_{L^2(\Omega)},
 \end{split}
 \end{equation*}
 where we used Proposition \ref{emb} and Lemma \ref{lemmaemb}.
 \end{proof}

Corollary \ref{cor:assp} tells us $(\LL^t_{r-1},\mathbf{\Psi}_{r-2})$ satisfy Assumption \ref{VQ} and now with Theorem \ref{thm:Th}, we have the final result. 
\begin{cor}
Let $\VV_h = \LL^t_{r-1}$ and  $\QQ_h = \mathbf{\Psi}_{r-2}$. Consider the problem \eqref{equ:CMod} with the nonzero eigenvalues  $0<\lambda^{(1)} \leq \lambda^{(2)} \leq ...$ and the problem (\ref{equ:DMod}) 
with the nonzero eigenvalues $0<\lambda_h^{(1)} \leq \lambda_h^{(2)} \leq ...$. Then, for any fixed $i$, $\lim_{h \rightarrow 0} \lambda_h^{(i)} = \lambda^{(i)}$.
\end{cor}

\section{Numerical Experiments}\label{sec-numerics}
In this section we provide numerical experiments which support our theoretical work.  All the computations were carried out using FEniCS \cite{alnaes2015fenics}. We consider the domain $\Omega = (0,\pi)^3$, \MJN{so that the exact eigenvectors of problem \eqref{equ:Ceig} (with non-zero eigenvalues)} are of the following form:
 \begin{equation}
         \bu= \begin{pmatrix}
           a_1 \cos(k_1x)\sin(k_2y)\sin(k_3z) \\
           a_2 \sin(k_1x)\cos(k_2y)\sin(k_3z) \\
           a_3 \sin(k_1x)\sin(k_2y)\cos(k_3z)
          \end{pmatrix},\quad \MJN{k_j\in \mathbb{N}\cup\{0\},\ k_1+k_2+k_3\ge 2,}
 \end{equation}
 where the coefficient $a_1, a_2,a_3$ are determined by the divergence-free equation: $a_1 k_1 +a_2k+a_3k_3=0$ and the non-zero eigenvalues are of the form $\MJN{\eta^2 = \lambda} = k_1^2+k_2^2+\MJN{k_3^2}$. We can compute the first few eigenvalues explicitly: 2 (with multiplicity 3), 3 (with multiplicity 2), 5 (with multiplicity 6), 6 (with multiplicity 6). 
 
 We \MJN{compute problem \eqref{equ:Deig}} using 
 three different choices of $\VV_h$: (i) linear and quadratic Lagrange finite element on mesh without Worsey-Farin refinement, 
(ii) linear Lagrange finite element space on Worsey-Farin meshes and (iii) quadratic Lagrange finite element space on Worsey-Farin meshes.

Table \ref{PG} states the first $13$ computed eigenvalues 
using linear and quadratic elements on a mesh without Worsey-Farin refinement
with $h=\pi/8$.  The numerics clearly indicate
that the discrete eigenvalues are poor approximations with $O(1)$ errors.
Likewise, numerical experiments using the linear Lagrange finite element space 
on Worsey-Farin meshes lead to computed eigenvalues
that are far from the exact solution (cf.~Table \ref{P1WF}).

Finally, we report the computed eigenvalues 
of method \eqref{equ:Deig} using quadratic Lagrange elements
on Worsey-Farin meshes and report the first $13$ non-zero eigenvalues with $h=1/6$ in 
Table \ref{P2WF}.  As expected from the theoretical results,
this scenario leads to accurate approximate eigenvalues.
In addition the right column in Table \ref{P2WF} lists
the errors of the first computed (non-zero) eigenvalue and indicates
converges with at least cubic rate.


 
 \begin{table}[htbp]
 \centering
 \begin{tabular}{ccc}
 \hline
 $i$ & $\lambda_h$ & $|\lambda-\lambda_h|$ \\ \hline
 1 & 2.0610  & 6.10 $\times 10^{-2}$ \\
 2 & 2.0610  & 6.10 $\times 10^{-2}$ \\
 3 & 2.0774  & 7.74 $\times 10^{-2}$ \\
 4 & 2.0900  & 9.10 $\times 10^{-1}$ \\
 5 & 2.0900  & 9.10 $\times 10^{-1}$ \\
 6 & 2.1506  & 2.85  \\
 7 & 2.1506  & 2.85 \\
 8 & 2.2698  & 2.73 \\
 9 & 2.2698  & 2.73 \\
 10 & 2.2910  & 2.71 \\
 11 & 2.3304  & 2.67 \\
 12 & 2.3514  & 3.65 \\
 13 & 2.3514  & 3.65 \\\hline
 \end{tabular}
 \quad
 \begin{tabular}{ccc}
 \hline
 $i$ & $\lambda_h$ & $|\lambda-\lambda_h|$ \\ \hline
 1 & 2.6699  & 6.67 $\times 10^{-1}$ \\
 2 & 2.6766  & 6.77 $\times 10^{-1}$ \\
 3 & 2.7369  & 7.34 $\times 10^{-1}$ \\
 4 & 2.7510  & 2.49 $\times 10^{-1}$ \\
 5 & 2.7510  & 2.49 $\times 10^{-1}$ \\
 6 & 2.7615  & 2.24 \\
 7 & 2.7615  & 2.24 \\
 8 & 2.7782  & 2.22 \\
 9 & 2.7782  & 2.22  \\
 10 & 2.7941  & 2.21  \\
 11 & 2.8244  & 2.18  \\
 12 & 2.8244  & 3.17  \\
 13 & 2.8855  & 3.11  \\\hline
 \end{tabular}\caption{The case of linear (left) and quadratic (right) Lagrange finite element space on a mesh without Worsey-Farin refinement, $h=\pi/8$.}
 \label{PG}
 \end{table}



 \begin{table}[htbp]
 \centering
 \begin{tabular}{ccc}
 \hline
 $i$ & $\lambda_h$ & $|\lambda-\lambda_h|$ \\ \hline
 1 & 2.6763  & 6.77 $\times 10^{-1}$ \\
 2 & 2.6763  & 6.77 $\times 10^{-1}$ \\
 3 & 2.6775  & 6.78 $\times 10^{-1}$ \\
 4 & 2.6775  & 3.23 $\times 10^{-1}$ \\
 5 & 2.7112  & 2.89 $\times 10^{-1}$ \\
 6 & 2.7825  & 2.22 \\
 7 & 2.7825  & 2.22 \\
 8 & 2.7860  & 2.21 \\
 9 & 2.8254  & 2.17  \\
 10 & 2.8878  & 2.11  \\
 11 & 2.9440  & 2.06  \\
 12 & 2.9440  & 3.56  \\
 13 & 2.9881  & 3.01  \\\hline
 \end{tabular}\caption{The case of linear Lagrange finite element space on a Worsey-Farin mesh, $h=\pi/10$.}
 \label{P1WF}
 \end{table}


 \begin{table}[htbp]
 \centering
 \begin{tabular}{ccc}
 \hline
 $i$ & $\lambda_h$ & $|\lambda-\lambda_h|$ \\ \hline
 1 & 2.000121  & 1.21 $\times 10^{-4}$ \\
 2 & 2.000243  & 2.43 $\times 10^{-4}$ \\
 3 & 2.000243  & 2.43 $\times 10^{-4}$ \\
 4 & 3.000741  & 7.41 $\times 10^{-4}$ \\
 5 & 3.000741  & 7.41 $\times 10^{-4}$ \\
 6 & 5.001307  & 1.31 $\times 10^{-3}$ \\
 7 & 5.001307  & 1.31 $\times 10^{-3}$ \\
 8 & 5.001862  & 1.86 $\times 10^{-3}$ \\
 9 & 5.002382  & 2.38 $\times 10^{-3}$ \\
 10 & 5.002913  & 2.91 $\times 10^{-3}$ \\
 11 & 5.002913  & 2.91 $\times 10^{-3}$ \\
 12 & 6.001822  & 1.82 $\times 10^{-3}$ \\
 13 & 6.002854  & 2.85 $\times 10^{-3}$ \\\hline
 \end{tabular}
 \quad
 \begin{tabular}{ccc}
 \hline
 $h$ & $|\lambda^{(1)}-\lambda_h^{(1)}|$ & Rate \\ \hline
 $\pi/5$ &  1.95 $\times 10^{-4}$ & / \\
 $\pi/6$ &  1.21 $\times 10^{-4}$ & 2.62 \\
 $\pi/7$ &  7.40 $\times 10^{-5}$ & 3.19 \\
 $\pi/8$ &  4.64 $\times 10^{-5}$ & 3.48 \\
 $\pi/9$ &  3.03 $\times 10^{-5}$ & 3.62 \\\hline
 \end{tabular} \caption{The case of $P_2$ finite element space with Worsey-Farin meshes}
 \label{P2WF}
 \end{table}

\SG{ \section{Conclusion}
In this paper, we studied and justified the convergence theory of 
the three-dimensional Maxwell eigenvalue problem using Lagrange finite element spaces
on Worsey-Farin splits. 
Although we only focus on Worsey-Farin splits in this paper, we provide a framework of proof which may apply to other refinements if we could fit the spaces into a de Rham complex. 
}




 \bibliographystyle{siam}
 \bibliography{ref}

\begin{thebibliography}{10}

\bibitem{alnaes2015fenics}
{\sc M.~Aln{\ae}s, J.~Blechta, J.~Hake, A.~Johansson, B.~Kehlet, A.~Logg,
  C.~Richardson, J.~Ring, M.~E. Rognes, and G.~N. Wells}, {\em The {FE}ni{CS}
  {P}roject {V}ersion 1.5}, Archive of Numerical Software, 3 (2015).

\bibitem{VP1998}
{\sc C.~Amrouche, C.~Bernardi, M.~Dauge, and V.~Girault}, {\em Vector
  {P}otential in {T}hree-dimensional {N}on-smooth {D}omains}, Mathematical
  Methods in the Applied Sciences, 21 (1998), pp.~823--864.

\bibitem{FE2006}
{\sc D.~N. Arnold, R.~S. Falk, and R.~Winther}, {\em Finite element exterior
  calculus,homological techniques, and applications}, Acta Numerica,  (2006),
  pp.~1--155.

\bibitem{badia2012nodal}
{\sc S.~Badia and R.~Codina}, {\em A nodal-based finite element approximation
  of the {M}axwell problem suitable for singular solutions}, SIAM Journal on
  Numerical Analysis, 50 (2012), pp.~398--417.

\bibitem{FE2010}
{\sc D.~Boffi}, {\em Finite {E}lement approximation of eigenvalue problems},
  Acta Numerica,  (2010), pp.~1--120.

\bibitem{CME1999}
{\sc D.~Boffi, P.~Fernandes, L.~Gastaldi, and I.~Perugia}, {\em Computational
  {M}odels of {E}lectromagnetic {R}esonators: {A}nalysis of {E}dge {E}lement
  {A}pproximation}, SIAM Journal on Numerical Analysis, 36 (1999),
  pp.~1264--1290.

\bibitem{CLM2021}
{\sc D.~Boffi, J.~Guzman, and M.~Neilan}, {\em Convergence of {L}agrange finite
  elements for the {M}axwell {E}igenvalue {P}roblem in 2d}, IMA Journal of
  Numerical Analysis,  (2022).
\newblock to appear.

\bibitem{bonito2011approximation}
{\sc A.~Bonito and J.-L. Guermond}, {\em Approximation of the eigenvalue
  problem for the time harmonic {M}axwell system by continuous {L}agrange
  finite elements}, Mathematics of Computation, 80 (2011), pp.~1887--1910.

\bibitem{brenner2008mathematical}
{\sc S.~C. Brenner and L.~R. Scott}, {\em The mathematical theory of finite
  element methods}, vol.~3, Springer, 2008.

\bibitem{buffa2009solving}
{\sc A.~Buffa, P.~Ciarlet, and E.~Jamelot}, {\em Solving electromagnetic
  eigenvalue problems in polyhedral domains with nodal finite elements},
  Numerische Mathematik, 113 (2009), pp.~497--518.

\bibitem{christiansen2018generalized}
{\sc S.~H. Christiansen and K.~Hu}, {\em Generalized finite element systems for
  smooth differential forms and stokes’ problem}, Numerische Mathematik, 140
  (2018), pp.~327--371.

\bibitem{PC2013}
{\sc P.~Ciarlet}, {\em Analysis of the {S}cott {Z}hang interpolation in the
  fractional order {S}obolev spaces}, Journal of Numerical Mathematics, 21
  (2013), pp.~173--180.

\bibitem{clough1965finite}
{\sc R.~W. Clough}, {\em Finite element stiffness matricess for analysis of
  plate bending}, in Proc. of the First Conf. on Matrix Methods in Struct.
  Mech., 1965, pp.~515--546.

\bibitem{ON2010}
{\sc M.~Costabel and A.~McIntosh}, {\em On {B}ogovski{\u\i} and regularized
  {P}oincar{\'e} integral operators for de {R}ham complexes on {L}ipschitz
  domains}, Mathematische Zeitschrift, 265 (2010), pp.~297--320.

\bibitem{IR2018}
{\sc I.~Drelichman and R.~G. Durán}, {\em Improved {P}oincaré inequalities in
  fractional {S}obolev spaces}, Annales Academiæ Scientiarum Fennicæ
  Mathematica, 43 (2018), p.~885–903.

\bibitem{du2020mixed}
{\sc Z.~Du and H.~Duan}, {\em A {M}ixed {M}ethod for {M}axwell {E}igenproblem},
  Journal of Scientific Computing, 82 (2020), pp.~1--37.

\bibitem{duan2019new}
{\sc H.~Duan, Z.~Du, W.~Liu, and S.~Zhang}, {\em New mixed elements for
  {M}axwell equations}, SIAM Journal on Numerical Analysis, 57 (2019),
  pp.~320--354.

\bibitem{duan2019family}
{\sc H.~Duan, W.~Liu, J.~Ma, R.~C. Tan, and S.~Zhang}, {\em A family of optimal
  {L}agrange elements for {M}axwell’s equations}, Journal of Computational
  and Applied Mathematics, 358 (2019), pp.~241--265.

\bibitem{Elliptic2011}
{\sc P.~Grisvard}, {\em Elliptic Problems in Nonsmooth Domains}, Society for
  Industrial and Applied Mathematics, Philadelphia, USA, 2011.

\bibitem{guzman2020exact}
{\sc J.~Guzm{\'a}n, A.~Lischke, and M.~Neilan}, {\em Exact sequences on
  {P}owell--{S}abin splits}, Calcolo, 57 (2020), pp.~1--25.

\bibitem{Ex2020}
{\sc J.~Guzman, A.~Lischke, and M.~Neilan}, {\em Exact sequences on
  {W}orsey-{F}arin splits}, Mathematics of Computation,  (2022).

\bibitem{frac2014}
{\sc N.~Heuer}, {\em On the equivalence of fractional-order {S}obolev
  semi-norms}, Journal of Mathematical Analysis and Applications, 417 (2014),
  pp.~2505--518.

\bibitem{hu2022partially}
{\sc J.~Hu, K.~Hu, and Q.~Zhang}, {\em Partially discontinuous nodal finite
  elements for $\mathrm{H} (\rm {curl})$ and $\mathrm{H} (\rm {div})$}, arXiv
  preprint arXiv:2203.02103,  (2022).

\bibitem{hu2021spurious}
{\sc K.~Hu, Q.~Zhang, J.~Han, L.~Wang, and Z.~Zhang}, {\em Spurious solutions
  for high order curl problems}, arXiv preprint arXiv:2110.12481,  (2021).

\bibitem{lai2007spline}
{\sc M.-J. Lai and L.~L. Schumaker}, {\em Spline functions on triangulations},
  vol.~110, Cambridge University Press, 2007.

\bibitem{FEM2003}
{\sc P.~Monk}, {\em Finite Element Methods for Maxwell's Equations}, Clarendon
  Press, Oxford, England, 2003.

\bibitem{MFE1980}
{\sc J.~C. N\'ed\'elec}, {\em Mixed finite elements in ${R}^3$}, Numerische
  Mathematik, 35 (1980), pp.~315--341.

\bibitem{powell1977piecewise}
{\sc M.~J. Powell and M.~A. Sabin}, {\em Piecewise quadratic approximations on
  triangles}, ACM Transactions on Mathematical Software (TOMS), 3 (1977),
  pp.~316--325.

\bibitem{SZ1990}
{\sc L.~R. Scott and S.~Zhang}, {\em Finite element interpolation of non smooth
  functions satisfying boundary conditions}, Mathematics of Computation, 54
  (1990), pp.~483--493.

\bibitem{walkington2014c}
{\sc N.~J. Walkington}, {\em A ${C}^1$ {T}etrahedral {F}inite {E}lement without
  {E}dge {D}egrees of {F}reedom}, SIAM Journal on Numerical Analysis, 52
  (2014), pp.~330--342.

\bibitem{wong1988combined}
{\sc S.~H. Wong and Z.~Cendes}, {\em Combined finite element-modal solution of
  three-dimensional eddy current problems}, IEEE Transactions on Magnetics, 24
  (1988), pp.~2685--2687.

\bibitem{worsey1987ann}
{\sc A.~Worsey and G.~Farin}, {\em An n-dimensional {C}lough-{T}ocher
  interpolant}, Constructive Approximation, 3 (1987), pp.~99--110.

\end{thebibliography}

 \appendix

\section{Proof of Theorem \ref{thm:Th}}
Before we prove Theorem \ref{thm:Th} we will need a few lemmas. 
\begin{lemma}\label{AT}
With Assumption \ref{VQ}, there exists a positive constant $C$ such that 
\begin{equation*}
    \|\bA\bbP_Q {\bm f}-\bA {\bm f}\|_{L^2(\Omega)}+\|\bT\bbP_Q{\bm f}-\bT {\bm f}\|_{L^2(\Omega)} \leq C \omega_1(h)\|{\bm f}\|_{L^2(\Omega)}~~~\forall {\bm f} \in \bL^2(\Omega).
\end{equation*}
\end{lemma}
\begin{proof}
Let ${\bm f} \in \bL^2(\Omega)$ and set $\bsigma=\A \f$, $\bu=\T\f$,$\bpsi=\A\bbP_Qf$ and $\bw=\T \bbP_Q \f$. With $\f$ and $\bbP_Q \f$ in (\ref{CSour}), we see that 
\begin{equation}
    \begin{split}
        (\bsigma-\bpsi,\btau)+(\bu-\bw,\bcurl\btau)=0~~~~~~~~~& \forall\btau \in \bH_0(\mathbf{curl}, \Omega), \\
        (\bcurl (\bsigma-\bpsi),\bq)=({\bm f}-\bbP_Q \f,\bq)~~~~~~ & \forall \bq \in \bH_0({\rm div}^0, \Omega).
    \end{split}
\end{equation}
Setting $\bq=\bw-\bu$ and $\btau=\bsigma-\bpsi$ in the above equations gives $\|\bsigma-\bpsi\|^2_{L^2(\Omega)}=({\bm f}-\bbP_Q \f,\bw-\bu)$, and the above equations also tell us that $\bcurl (\bu-\bw)=\bsigma-\bpsi$. Furthermore, $\bu-\bw \in \bH(\mathbf{curl}, \Omega) \cap \bH_0({\rm div}^0, \Omega)$ and there holds
\begin{equation*}
    \|\bsigma-\bpsi\|_{L^2(\Omega)} \leq \sup\limits_{\boldsymbol{\phi} \in \bH(\mathbf{curl}, \Omega) \cap \bH_0({\rm div}^0, \Omega)} \frac{(\f-\bbP_Q \f,\boldsymbol{\phi})}{\|\bcurl \boldsymbol{\phi}\|_{L^2(\Omega)}}.
\end{equation*}
Moreover, Assumption \ref{VQ} gives us
\begin{alignat*}{1}
    \sup\limits_{\boldsymbol{\phi} \in \bH(\mathbf{curl}, \Omega) \cap \bH_0({\rm div}^0, \Omega)} \frac{(\f-\bbP_Q \f,\boldsymbol{\phi})}{\|\bcurl \boldsymbol{\phi}\|_{L^2(\Omega)}} =  &\sup\limits_{\boldsymbol{\phi} \in \bH(\mathbf{curl}, \Omega) \cap \bH_0({\rm div}^0, \Omega)} \frac{(\f,\boldsymbol{\phi}-\bbP_Q \boldsymbol{\phi})}{\|\bcurl \boldsymbol{\phi}\|_{L^2(\Omega)}} \\
    \leq &\,  \omega_1(h) \|\f\|_{L^2(\Omega)}.
\end{alignat*}

Thus, we have shown 
\begin{equation*}
    \|\bA\bbP_Q {\bm f}-\boldsymbol{Af}\|_{L^2(\Omega)} \leq  \omega_1(h)\|{\bm f}\|_{L^2(\Omega)}
\end{equation*}
On the other hand, since $\bT(\bbP_Q {\bm f} -\f) \in \bH(\mathbf{curl}, \Omega) \cap \bH_0({\rm div}^0, \Omega)$, we have by Lemma \ref{emb},
\begin{equation*}
\begin{split}
    \|\bT\bbP_Q {\bm f}-\boldsymbol{Tf}\|_{L^2(\Omega)} & \leq C\|\bcurl(\bT\bbP_Q {\bm f}-\boldsymbol{Tf})\|_{L^2(\Omega)} \\ 
    &= C\|\bA\bbP_Q {\bm f}-\boldsymbol{Af}\|_{L^2(\Omega)} \leq C \omega_1(h)\|{\bm f}\|_{L^2(\Omega)}.
\end{split}
\end{equation*}
\end{proof}

Since the domain $\Omega$ is contractible, then the set of harmonic forms is trivial and 
we have (cf.~\cite[Theorem 4.9]{ON2010}): 
\begin{lemma} \label{ss}
Let $\Omega$ be a bounded, contractible,  Lipschitz domain in $\mathbb{R}^3$. Then for all $\bu \in \bH_0({\rm div}^0,\Omega)$,  there exists $\bv \in \bH^{1}_0(\Omega)$ such that $\bu=\bcurl \bv$ and $\|\bv\|_{H^{1}(\Omega) }\leq C\|\bu\|_{L^2(\Omega)}$.
\end{lemma}

\begin{lemma} \label{IS}
With Assumption \ref{VQ}, there exists a positive constant $C$ such that for every $\bp_h \in \QQ_h$, there exists $\btau_h \in \VV_h$ such that $\bcurl\btau_h = \bp_h$ and $\|\btau_h\|_{L^2(\Omega)} \leq C\|\bp_h\|_{L^2(\Omega)}$.
\end{lemma}
\begin{proof}
By Lemma \ref{ss}, for every $\bp_h \in \QQ_h \subset \bH_0({\rm div}^0, \Omega)$, there exists $\btau \in \bH^1_0(\Omega)$ such that $\bcurl\btau = \bp_h$ and $\|\btau\|_{H^1(\Ome)}\le C \|\bp_h\|_{L^2(\Omega)}$. 
Note that since $\btau \in \VV^t(\QQ_h)$, we can set $\btau_h=\bPi_{\VV}\btau$. By  Assumption \ref{VQ}, we have $\bcurl\btau_h=\bcurl\btau=\bp_h$. Additionally, 
\begin{equation*}
    \|\btau_h\|_{L^2(\Omega)} \leq C (\|\btau\|_{H^{1/2+\delta}(\Omega)} + \|\bcurl\btau\|_{L^2(\Omega)} ) \leq C \|\btau\|_{H^1(\Omega)} \leq C \|\bp_h\|_{L^2(\Omega)}.
\end{equation*}
\end{proof}

We are now in position to prove Theorem \ref{thm:Th}.
\begin{proof}[Proof of Theorem \ref{thm:Th}.]
Let $\f \in \bL^2(\Omega)$ and set $\bsigma=\A \f$, $\bu=\T\f$,$\bsigma_h=\A_h \f$, $\bu_h=\T_h\f$, $\bpsi=\A\bbP_Qf$ and $\bw=\T \bbP_Q \f$. The first step is to estimate $\bpsi-\bsigma_h$. We note that $\bcurl \bpsi= \bbP_Q \f$ and $\bcurl \bw= \bpsi$.  From this we have that $\bpsi \in \VV^t(\QQ_h)$, $\Div \bpsi=0$ and $\bw \in   \bH(\bcurl, \Omega) \cap \bH_0({\rm div}^0, \Omega)$. Hence, by Assumption 1 we get $ \bcurl (\bPi_{\VV} \bpsi)=\bcurl \bpsi= \bbP_Q \f$. We can write the error equations
\begin{subequations}
\begin{alignat}{2}
    (\bPi_{\VV} \bpsi-\bsigma_h,\btau_h)+(\bbP_Q\bw- \bu_h,\bcurl\btau_h) & =(\bPi_{\VV} \bpsi-\bpsi,\btau_h) \quad \quad && \forall\btau_h \in \VV_h, \label{eq712a}\\
     (\bcurl (\bPi_{\VV} \bpsi-\bsigma_h),\bq_h)& =0 && \forall \bq_h \in \QQ_h. \label{eq712b}
\end{alignat}
\end{subequations}
Setting $\btau_h=\bPi_{\VV} \bpsi-\bsigma_h$ and using the Cauchy-Schwarz inequality provides:
\begin{equation}\label{eq715}
    \|\bPi_{\VV} \bpsi-\bsigma_h\|_{L^2(\Omega)} \leq \|\bPi_{\VV} \bpsi-\bpsi\|_{L^2(\Omega)} \leq \omega_0(h)(\|\bpsi\|_{H^{1/2+\delta}(\Omega)} + \|\bcurl \bpsi\|_{L^2(\Omega)} ).
\end{equation}
Next, with the embedding result in Proposition \ref{emb} and noting that $\Div \bpsi=0$, we obtain
\begin{equation}\label{eq714}
\begin{split}
    \|\bpsi\|_{H^{1/2+\delta}(\Omega)} & \leq C(\|\bcurl \bpsi\|_{L^2(\Omega)}+\|\Div \bpsi\|_{L^2(\Omega)}) \\  
    & \leq C\|\bcurl \bpsi\|_{L^2(\Omega)} = C\|\bbP_Q \f\|_{L^2(\Omega)}  \leq C\|\f\|_{L^2(\Omega)}.
\end{split}
\end{equation}
Thus, we have $\|\bPi_{\VV} \bpsi-\bsigma_h\|_{L^2(\Omega)} \leq C\omega_0(h)\|\f\|_{L^2(\Omega)}$.
Lemma \ref{AT} shows that
\begin{equation}\label{eq713}
     \|\bsigma-\bpsi\|_{L^2(\Omega)}+\|\bw-\bu\|_{L^2(\Omega)} \leq C \omega_1(h)\|{\bm f}\|_{L^2(\Omega)},
\end{equation}
and therefore, combining \SG{\eqref{eq715}, \eqref{eq714}, \eqref{eq713}} we obtain 
\begin{equation*}
    \begin{split}
        \|(\A-\A_h) \f\|_{L^2(\Omega)}& =\|\bsigma-\bsigma_h\|_{L^2(\Omega)}\\
        & \leq \|\bsigma-\bpsi\|_{L^2(\Omega)}+\|\bsigma_h-\bPi_{\VV} \bpsi\|_{L^2(\Omega)}+\|\bPi_{\VV} \bpsi-\bpsi\|_{L^2(\Omega)} \\
        & \leq C (\omega_0(h)+\omega_1(h))\|{\bm f}\|_{L^2(\Omega)}.
    \end{split}
\end{equation*}

By Lemma \ref{IS}, we know that 
there exists $\btau_h \in \VV_h$ such that $\bcurl\btau_h = \bbP_Q \bw-\bu_h$ and  $\|\btau_h\|_{L^2(\Omega)}\le C \SG{\|\bbP_Q \bw-\bu_h\|_{L^2(\Omega)}}$.
Then using \eqref{eq712a} and the Cauchy-Schwarz inequality, we have
\begin{equation*}
    \|\bbP_Q \bw-\bu_h\|_{L^2(\Omega)}^2=(\bbP_Q \bw-\bu_h, \bcurl \btau_h)= -(\bpsi-\bsigma_h,\btau_h) \leq C\|\bbP_Q \bw-\bu_h\|_{L^2(\Omega)}\|\bpsi-\bsigma_h\|_{L^2(\Omega)}.
\end{equation*}
Furthermore, using the triangle inequality, and Assumption 1,  we have
\begin{alignat}{1}
        \|\bw-\bu_h\|_{L^2(\Omega)} &\leq \|\bbP_Q \bw-\bu_h\|_{L^2(\Omega)}+\|\bw-\bbP_Q \bw\|_{L^2(\Omega)}  \nonumber \\
        & \leq C\|\bpsi-\bsigma_h\|_{L^2(\Omega)}+\omega_1(h)\|\bcurl \bw\|_{L^2(\Omega)} \nonumber \\
        & \SG{\leq C (\omega_0(h)+\omega_1(h))\|{\bm f}\|_{L^2(\Omega)}}+ \omega_1(h)\|\bcurl \bw\|_{L^2(\Omega)}. \label{eq717}
\end{alignat}
We also have using Lemma \ref{lemmaemb}, 
\begin{equation*}
\|\bcurl \bw\|_{L^2(\Omega)}=\| \bpsi\|_{L^2(\Omega)} \le C \| \bcurl \bpsi\|_{L^2(\Omega)}= C\|\bbP_Q \f\|_{L^2(\Omega)} \leq C\|\f\|_{L^2(\Omega)},
\end{equation*}
and therefore using this inequality with \eqref{eq717} and \eqref{eq713} we arrive at 
\begin{equation*}
\begin{split}
      \|(\T-\T_h)\f\|_{L^2(\Omega)} & =\|\bu-\bu_h\|_{L^2(\Omega)} \leq \|\bu-\bw\|_{L^2(\Omega)} +\|\bw-\bu_h\|_{L^2(\Omega)}\\
      & \leq C (\omega_0(h)+\omega_1(h))\|{\bm f}\|_{L^2(\Omega)}.
\end{split}
\end{equation*}
\end{proof}

\section{Proof of Lemma \ref{lem: SZ}}
Before proving Lemma \ref{lem: SZ} let us develop some \SG{notations}. Let  
$\{E_1, \ldots, E_M\}$ 
and $\{ v_1, \ldots, v_m\}$ be the sets of edges and (corner) vertices, respectively, of the polyhedral domain $\Omega$.  For each $E_j$ there exists two faces of $\partial \Omega$, $f_j^1$ and $f_j^2$ that share $E_j$, and we denote their unit-normal vectors by $\bn_j^i$ for $i=1,2$. 
We let $\tilde{\bT}_j^3$ be tangent to $E_j$ and set  $\tilde{\bT}_j^i= \bn_j^i \times  \tilde{\bT}_j^3$ for $i=1,2$. \SG{Note that $\tilde{\bT}_j^i$ for $i=1,2,3$ are linearly independent.} 

Let $\mathcal{V}_h^j$ denote all the vertices of $\mathcal{T}_h^{wf}$ that are in the interior of $E_j$. Recall that $\mathcal{V}_h$ is the set of vertices of $\mathcal{T}_h^{wf}$. We decompose them in the following form 
\begin{equation*}
\mathcal{V}_h= \mathcal{V}_h^C \cup \mathcal{V}_h^E \cup \mathcal{V}_h^0,   
\end{equation*}
where $\mathcal{V}_h^C=\{ v_1, \ldots, v_m\}$ are the corner points and 
$\mathcal{V}_h^E= \cup_{1 \le j \le M } \SG{\mathcal{V}_h^j}$ 
are the vertices lying on edges of $\partial \Omega$. Finally,  $\mathcal{V}_h^0= \mathcal{V}_h \backslash (\mathcal{V}_h^C \cup \mathcal{V}_h^E)$. 

For any $z \in \mathcal{V}_h^j$ we choose $F_z^i \in \Delta_2(\mathcal{T}_h^{wf})$ such that $z$ is a vertex of $F_z^i$ and $F_z^i \subset f_j^i$ for $i=1,2$. We then set $F_z^3=F_z^2$. We also set $\bT_z^i= \tilde \bT_j^i$ for $i=1,2,3$.  If $z \in \mathcal{V}_h^C$ then $z$ is an end point  of some $E_j$ and we define $F_z^i$  and $\bT_z^i$ for $i=1,2,3$ in the same way.

 \begin{proof}[Proof of Lemma \ref{lem: SZ}.]
 First, we summarize the construction of the 
Scott-Zhang interpolant in \cite{SZ1990}. Let $\mathcal{V}_h$ be the set of all the vertices
of $\mathcal{T}_h$.  For each $z \in \mathcal{V}_h$, let $\phi_z$ be corresponding nodal basis of $\mathcal{P}_1^c(\mathcal{T}_h)$, i.e., $\phi_z \in \mathcal{P}_1^c(\mathcal{T}_h)$ satisfies $\phi_z(y)=\delta_{yz}$ for all $y \in \mathcal{V}_h$. For every $z \in \mathcal{V}_h$, 
we identify an arbitrary face $F_z$ of the mesh that contains $z$ with the only constraint that $F_z$ is a boundary face if $z$ is a boundary vertex. 
Then there exists function $\psi_z \in L^{\infty}(F_z)$ such that
 \begin{equation}\label{SZorth}
     \int_{F_z} \psi_z \phi_y = \delta_{yz},~~~~~ \forall y \in \mathcal{V}_h.
 \end{equation}
 Furthermore, the function $\psi$ satisfies the estimate:
 \begin{equation} \label{SZsigma}
     \|\psi_z\|_{L^{\infty}(F_z)} \leq \frac{C}{|F_z|}.
 \end{equation}
 The Scott-Zhang interpolant $\boldsymbol{\tilde{I}}_h$ is given by:
 \begin{equation}
     \boldsymbol{\tilde{I}}_h\btau(x)=\sum\limits_{z \in \mathcal{V}_h} (\int_{F_z}\psi_z\btau)\phi_z(x).
 \end{equation}

\noindent{\bf Construction of $\bI_h^{curl}$:}
 Similar to the construction in \cite{CLM2021} for
 the two-dimensional case,
 we modify the Scott-Zhang interpolant $\tilde{\bm I}_h$ on edges
 and corner vertices of $\Omega$ to preserve the vanishing tangential trace.  We also let $\psi_z^i$ (for $i=1,2,3$) satsify:
 \begin{equation}
         \int_{F_z^i} \psi_z^i \phi_y = \delta_{yz},~~~~~ \forall y \in \mathcal{V}_h,~i=1,2,3
 \end{equation}
 \begin{equation} \label{SZf}
     \|\psi_z^i\|_{L^{\infty}(F_z^i)} \leq \frac{C}{|F_z^i|}.
 \end{equation}
 The modified Scott-Zhang interpolant $\bI_h^{curl}$ is given as 
 \begin{equation} \label{SZcurl}
     \boldsymbol{I}_h^{curl}\btau(x)=\sum\limits_{z \in \mathcal{V}_h^0} (\int_{F_z}\psi_z\btau)\phi_z(x) + \sum \limits_{z \in \mathcal{V}_h^C \cup \mathcal{V}_h^E} \boldsymbol{\beta}^t_z(\btau) \phi_z(x)
 \end{equation}
 where 
 \begin{equation*}
     \boldsymbol{\beta}^t_z(\btau):= \sum\limits_{i=1}^3 \frac{\boldsymbol{C}_i}{(\bT_z^1 \times \bT_z^2) \cdot \bT_z^3} \int_{F_z^i}(\btau\cdot\bT_z^i) \psi_z^i,
 \end{equation*}
 with $\boldsymbol{C}_1=\bT_z^2 \times \bT_z^3$,$\boldsymbol{C}_2=-\bT_z^1 \times \bT_z^3$ and $\boldsymbol{C}_3=\bT_z^1 \times \bT_z^2$. Note that $\boldsymbol{C}_i \cdot \bT_z^{\ell} = \delta_{i\ell }$.

\noindent{\bf Construction of $\bI_h^{div}$:}
If $z \in \mathcal{V}_h^E$  we define $F_z^i$ for $i=1,2,3$ as above. Moreover, we set $\bn_z^i=\bn_j^i$ for $i=1,2$ and $\bn_z^3=\bT_z^3$. On the other hand, if \SG{$z \in \mathcal{V}_h^C$} we let $F_z^i \in \Delta_2(\mathcal{T}_h^{wf})$ for $i=1,2,3$ be such that $F_z^i \subset \partial \Omega$ with each $z$ being a vertex of $F_z^i$, and such that each \SG{lies} on a distinct plane. We then let $\bn_z^i$ be unit normal vectors to $F_z^i$.  
We then define
 \begin{equation} \label{SZdiv}
     \boldsymbol{I}_h^{div}\btau(x)=\sum\limits_{z \in \mathcal{V}_h^0} (\int_{F_z}\psi_z\btau)\phi_z(x) + \sum\limits_{z \in \mathcal{V}_h^E \bigcup \mathcal{V}_h^C} \boldsymbol{\beta}^n_z(\btau) \phi_z(x),
 \end{equation}
 where 
 \begin{equation*}
     \boldsymbol{\beta}^n_z(\btau):= \sum\limits_{i=1}^3 \frac{\boldsymbol{D}_i}{(\bn_z^1 \times \bn_z^2) \cdot \bn_z^3} \int_{F_z^i}(\btau\cdot\bn_z^i) \psi_z^i,
 \end{equation*}
 with $\boldsymbol{D}_1=\bn_z^2 \times \bn_z^3$, $\boldsymbol{D}_2=-\bn_z^1 \times \bn_z^3$ and $\boldsymbol{D}_3=\bn_z^1 \times \bn_z^2$. Note that $\boldsymbol{D}_i \cdot \bn_z^\ell = \delta_{i\ell}$.

\noindent{\bf Proof of estimates \eqref{SZest}--\eqref{SZestd}:}
We prove the estimates in four steps:
\begin{itemize}
\item[(ia)]  $\boldsymbol{I}_h^{curl}: \bH^{1/2+\delta}(\Omega) \cap \bH_0(\mathbf{curl},\Omega) \rightarrow \boldsymbol{\mathcal{P}}_1^c(\mathcal{T}_h) \cap \bH_0(\mathbf{curl},\Omega)$: 
 if $\btau \in \bH^{1/2+\delta}(\Omega) \cap \bH_0(\mathbf{curl},\Omega)$, then $\btau \times \bn |_{\partial \Omega} = 0$,
 and so 
 $(\btau \cdot \bT)(z)=0$ for any tangential vector $\bT$ at $z \in \partial \Omega$. Therefore, for every $z \in \mathcal{V}_h^E \bigcup \mathcal{V}_h^C$, $\boldsymbol{I}_h^{curl}\btau(z)=\boldsymbol{\beta}^t_z(\btau)=0$. 
 On the other hand, for every $z \in \mathcal{V}_h^0 \cap \partial \Omega$, we have $\boldsymbol{I}_h^{curl}\btau(z) \times \bn_{F_z}=\int_{F_z}\psi_z (\btau \times \bn_{F_z})=0$, where $\bn_{F_z}$ is the outward normal vector of $F_z \subset \partial \Omega$. 
These two identities yield $(\bI_h^{curl} \btau \times \bn)|_{\p\Omega} = 0$.
 
\item[(ib)] $\boldsymbol{I}_h^{div}: \bH^{1/2+\delta}(\Omega) \cap \bH_0({\rm div},\Omega) \rightarrow \boldsymbol{\mathcal{P}}_1^c(\mathcal{T}_h) \cap \bH_0({\rm div},\Omega)$: if $\btau \in \bH^{1/2+\delta}(\Omega) \cap \bH_0(\Div,\Omega)$, then $\btau \cdot \bn |_{\partial \Omega} = 0$. Suppose $z \in \mathcal{V}_h^E$, then the definition of $\bI_h^{div}$ shows $\bI_h^{div} \btau(z) \cdot \bn_{z}^i= \bbeta^n_z(\btau)\cdot \bn_z^i=0$ for $i=1,2$. On the other hand, if $z \in \mathcal{V}_h^C$ then $\boldsymbol{I}_h^{div}\btau(z) \cdot \bn_z^i=\boldsymbol{\beta}^n_z(\btau) \cdot \bn_z^i=0$, $i=1,2,3$,
and so in this case $\bI_h^{div} \btau(z)= 0$. Finally, if  $z \in \mathcal{V}_h^0 \cap \partial \Omega$, we have $\boldsymbol{I}_h^{div}\btau(z) \cdot \bn_{F_z}=\int_{F_z}\psi_z (\btau \cdot \bn_{F_z})=0$, where $\bn_{F_z}$ is the outward normal vector of $F_z \subset \partial \Omega$. We conclude that $(\bI_h^{curl} \btau)\cdot \bn|_{\p \Omega} = 0$.
 \item[(ii)] $\boldsymbol{I}_h^{curl}$ and $\boldsymbol{I}_h^{div}$ \textit{are projections.}  
 We show that if $\btau_1\in \bpol_1(\calT_h)\cap \bH(\mathbf{curl},\Omega)$
 and $\btau_2\in \bpol_1(\calT_h)\cap \bH({\rm div},\Omega)$,
 then $\bI_h^{curl} \btau_1 = \btau_2$ and $\bI_h^{div} \btau_2 = \btau_2$.
 Since $\btau_i \in \boldsymbol{\mathcal{P}}_1^c(\mathcal{T}_h)$ ($i=1,2$), we can write
\[
\btau_i(x) = \sum_{z \in \calV_h} \btau_i(y) \phi_{y}(x).
\]
 If $z \in \mathcal{V}_h^0 $,  
 $\boldsymbol{I}_h^{curl}\btau_1(z) = \int_{F_z}\psi_z\btau_1 = \sum\limits_{y \in \mathcal{V}_h}\btau_1(y)\int_{F_z}\psi_z \phi_{y}=\btau_1(z)$ by (\ref{SZorth}).  
 Similarly, $\boldsymbol{I}_h^{div}\btau_2(z) =\btau_2(z)$. 
 However, if $z \in \mathcal{V}_h^E \bigcup \mathcal{V}_h^C$ then $\boldsymbol{I}_h^{curl}\btau_1(z) = \boldsymbol{\beta}^t_z(\btau_1)$ and 
 $\boldsymbol{I}_h^{div}\btau_2(z) = \boldsymbol{\beta}^n_z(\btau_2)$. 
 We have $\boldsymbol{\beta}_z^t(\btau_1) \cdot \bT_z^i = \int_{F_z^i}(\btau_1\cdot\bT_z^i) \psi_z^i =\btau_1(z)\cdot\bT_z^i$ for $i=1,2,3$. 
 Recalling these three tangential vectors are linearly independent, we conclude  $\boldsymbol{I}_h^{curl}\btau_1(z) =\btau_1(z)$.

 Similarly, since $\boldsymbol{\beta}_z^n(\btau_2) \cdot \bn_z^i = \int_{F_z^i}(\btau_2\cdot\bn_z^i) \psi_z^i =\btau_2(z)\cdot\bn_z^i$, $i=1,2,3$, we have $\boldsymbol{I}_h^{div}\btau_2(z) =\btau_2(z)$.
\item[(iii)] \textit{Stability estimate.} By an inverse estimate  \eqref{inv1} we have 
 \begin{equation*}
     |\boldsymbol{I}_h^{curl}\btau|_{H^{1/2+\delta}(T)} \leq Ch_T^{-1/2-\delta}\|\boldsymbol{I}_h^{curl}\btau\|_{L^2(T)}.
 \end{equation*}
 We will use the following trace inequality (see \cite[Proposition 3.1]{PC2013}; also follows from \eqref{traceinq}, \SG{\eqref{equ:inv}} and a scaling argument),
 \begin{equation}\label{856}
     \|\btau\|_{L^1(F_z)} \leq C(h_T^{\frac{1}{2}}\|\btau\|_{L^2(T)}+h_T^{1+\delta}|\btau|_{H^{\frac{1}{2}+\delta}(T)}), \quad  T \in \mathcal{T}_h, F_z \in \Delta_2(T).
 \end{equation}
 Since the number of edges and vertices  of $\partial \Omega$ is finite  we have  $M_t:= \max\limits_{z \in \mathcal{V}_h^E \bigcup \mathcal{V}_h^C} \frac{1}{(\bT_z^1 \times \bT_z^2) \cdot \bT_z^3 }$ is finite. 
 Using the $L^\infty$ estimates of $\psi_z$, \eqref{SZsigma}, \eqref{SZf}, \eqref{856} and the estimate $\|\phi_z\|_{L^2(T)}\le C h_T^{3/2}$ for $z\in \bar T$, we have
 \begin{equation*}
 \begin{split}
      \|\bI_h^{curl}\btau\|_{L^2(T)}  \leq & \sum\limits_{\substack{z \in \mathcal{V}_h^0 \\ z \in \Bar{T}}}\|\phi_z\|_{L^2(T)}\|\psi_z\|_{L^{\infty}(F_z)}\|\btau\|_{L^1(F_z)} \\
      &~~~~ +M_t \sum\limits_{\substack{z \in \mathcal{V}_h^E \bigcup \mathcal{V}_h^C \\ z \in \Bar{T}}} \|\phi_z\|_{L^2(T)}\sum\limits_{i=1}^3 (\|\psi_z^i\|_{L^{\infty}(F_z^i)}\|\btau\|_{L^1(F_z^i)}) \\
      & \leq C(1+M_t)\big(\|\btau\|_{L^2(\omega(T))}+h_T^{1/2+\delta}|\btau|_{H^{\frac{1}{2}+\delta}(\omega(T))}\big).
 \end{split}
 \end{equation*}
 Therefore, we conclude
 \begin{equation}\label{SZstab}
      h_T^{1/2+\delta}|\boldsymbol{I}_h^{curl}\btau|_{H^{1/2+\delta}(T)} + \|\boldsymbol{I}_h^{curl}\btau\|_{L^2(T)} \leq C(1+M_t)(\|\btau\|_{L^2(\omega(T))}+h_T^{1/2+\delta}|\btau|_{H^{1/2+\delta}(\omega(T))}).
 \end{equation}
 On the other hand, by following the same process, we obtain
 \begin{equation}\label{SZstabDiv}
           h_T^{1/2+\delta}|\boldsymbol{I}_h^{div}\btau|_{H^{1/2+\delta}(T)} + \|\boldsymbol{I}_h^{div}\btau\|_{L^2(T)} \leq C(1+M_n)(\|\btau\|_{L^2(\omega(T))}+h_T^{1/2+\delta}|\btau|_{H^{1/2+\delta}(\omega(T))}),
 \end{equation}
 where $M_n:= \max\limits_{z \in \mathcal{V}_h^E \bigcup \mathcal{V}_h^C} \frac{1}{(\bn_z^1 \times \bn_z^2) \cdot \bn_z^3 }$.
 
 \item[(iv)] \textit{Estimate (\ref{SZest})}: Let $\boldsymbol{\omega} = \frac{1}{|\omega(T)|} \int_{\omega(T)}\btau\, dx$, so that  
by the Poincare inequality (cf.~\cite[Section 4]{SZ1990} and \cite[Proposition 2.1]{IR2018})
 \begin{equation}\label{SZest4}
     \|\btau-\boldsymbol{\omega}\|_{L^2(\omega(T))} \leq Ch_T^{1/2+\delta} |\btau|_{H^{1/2+\delta}(\omega(T))}.
 \end{equation}

 Since $\boldsymbol{\omega}$ is a constant, there holds $\boldsymbol{I}_h^{curl}\boldsymbol{\omega}|_T=\boldsymbol{\omega}$.
 Using the estimate \eqref{SZest4} and the stability result \eqref{SZstab}, we obtain
 \begin{equation*}
     \begin{split}
         \|\bI_h^{curl}\btau-\btau\|_{L^2(T)} & =  \|\boldsymbol{I}_h^{curl}(\btau-\boldsymbol{\omega})-(\btau-\boldsymbol{\omega})\|_{L^2(T)} \\
         & \leq \SG{C}(1+M_t) (\|\btau-\boldsymbol{\omega}\|_{L^2(\omega(T))}+h_T^{1/2+\delta}|\btau|_{H^{1/2+\delta}(\omega(T))})\\
         & \leq C(1+M_t) h_T^{1/2+\delta}|\btau|_{H^{1/2+\delta}(\omega(T))}.
     \end{split}
 \end{equation*}
 Similarly, we have
 \begin{equation*}
     \begin{split}
         |\bI_h^{curl}\btau|_{H^{1/2+\delta}(T)} & =  \SG{|\boldsymbol{I}_h^{curl}(\btau-\boldsymbol{\omega})|_{H^{1/2+\delta}(T)}}
         \\
         & \leq C (1+M_t) (h_T^{-1/2-\delta}\|\btau-\boldsymbol{\omega}\|_{L^2(\omega(T))}+|\btau|_{H^{1/2+\delta}(\omega(T))})\\
         & \leq C(1+M_t)|\btau|_{H^{1/2+\delta}(\omega(T))}.
     \end{split}
 \end{equation*}
 These last two estimates complete the proof of (\ref{SZest}).
 By the same process, (but replacing the stability estimate \eqref{SZstab} with \eqref{SZstabDiv})
we obtain the estimate for $\boldsymbol{I}^{div}_h$ \eqref{SZestd}.
 \end{itemize}
 \end{proof}
 
 \section{Scaling Properties}
 In this section, \SG{we need to prove Proposition \ref{prop:dialT}, Lemma \ref{lem:trace}, 
 and Lemma \ref{lem:lift}. }
 \SG{\subsection{Proof of Proposition \ref{prop:dialT}}
 \begin{proof}
 The proof of (i) follows from the definition of $\hat T$
 and Definition \ref{def:calKDef}.
 The identities given in (ii) follow  from the chain rule, 
 and the scaling result in (iii) is a direct application of \cite[Lemma 2.9]{frac2014}.
\end{proof}}
 
\SG{To prove those two lemmas}, we  transform $\hat T$ to the standard reference tetrahedron $\tilde T$ with unit size.
 \begin{definition} \label{def:transK}
 Let $\tilde  T$ be the tetrahedron with vertices  $(0,0,0),(1,0,0),(0,1,0),(0,0,1)$.
 For any tetrahedra $\hat T$ 
 set $\phi_{\hat T}:\tilde T\to \hat T$ to be an affine diffeomorphism with
  \[
  \phi_{\hat T}(\tilde x) = B_{\hat T}\SG{\tilde x} +b_{\hat T}\qquad \tilde x\in \tilde T
  \]
for some $B_{\hat T}\in \mathbb{R}^{3\times 3}$ and $b_{\hat T}\in \mathbb{R}^3$.

 \end{definition}
 \begin{remark}\label{rem:MonkRegul}
 There holds (cf.~\cite[Page 80 and Lemma 5.10]{FEM2003})
 \begin{alignat*}{2}
 &|\det(B_{\hat T})| = \frac{|\hat T|}{|\tilde T|} = 6 |\hat T|,\quad \|B_{\hat T}\|\le h_{\hat T} = 1,\quad \|B_{\hat T}^{-1}\|\le \rho_{\hat T}^{-1}\le c_0.
 \end{alignat*}
 
 \end{remark}
 


\subsection{Proof of Lemma \ref{lem:trace}}
We require an intermediate result
to prove Lemma \ref{lem:trace}.
\begin{prop} \label{prop:frac}
 (Equivalence norms of Sobolev space)
 \begin{itemize}
 \item[(i)] $\forall \hat \bv \in \bH^s(\hat T)$ with $0<s<1$, we have
 \begin{equation} \label{equ:frac}
     C_1 |\hat \bv|_{H^{s}(\hat T)}\le  |\hat \bv\circ \phi_{\hat T}|_{H^{s}(\tilde T)} \le C_2|\hat \bv|_{H^{s}(\hat T)},
 \end{equation} 
 where $C_1$ and $C_2$ depends on $c_0$ (the shape regularity of $\calT_h$) and $s$.
 \item[(ii)] For all $\hat \kappa \in W^{s,q}(\hat e)$ with $sq<1$ and $0<s<1/2$, we have 
 \[ 
 C_3\|\hat \kappa\circ \phi_{\hat T}\|_{W^{s,q}(\tilde e)} \le \| \hat \kappa\|_{W^{s,q}(\hat e)},
 \]
 where $\tilde e = \phi_{\hat T}^{-1}(\hat e)$, and  $C_3$ depends on $c_0$, $s$ and $q$. 
\item[(iii)] For all $\hat \kappa \in W^{1,p}(\hat T)$, we have
 \begin{equation} \label{equ:sobolev}
    \|\hat \kappa\|_{W^{1,p}(\hat T)} \le  C_4  \|\kappa\circ \phi_{\hat T}\|_{W^{1,p}(\tilde T)}.
 \end{equation} 
Moreover, for any $\hat F \in \Delta_2(\hat T)$, 
  \begin{equation} \label{equ:psobolev}
    \|\hat \kappa\|_{W^{1,p}(\hat F)} \le  C_5  \| \hat \kappa\circ \phi_{\hat T}\|_{W^{1,p}(\tilde F)}
    \quad \forall \hat \kappa \in W^{1,p}(\hat F),
 \end{equation} 
 where $\tilde F = \phi_{\hat T}^{-1}(\hat F)$, and $C_4$ and $C_5$ depends on $c_0$ and $p$.
 \end{itemize}
 \end{prop}
 \begin{proof}
 \begin{itemize}
\item[(i)]
This estimate follows from \cite[Lemma 2.9]{frac2014} and Remark \ref{rem:MonkRegul}.
 \item[(ii)]  
 Let $\hat \kappa\in W^{s,q}(\hat e)$
 and to ease notation, set $\tilde \kappa = \hat \kappa \circ \phi_{\hat T}$.
 Recalling the definition of $W^{s,q}(\hat e)$
 and applying a change of variables, we have
  \begin{equation*}
  \begin{split}
      |\hat \kappa|^q_{W^{s,q}(\hat e)} & = \int_{\hat e} \int_{\hat e} \frac{|\hat \kappa(\hat x)-\hat \kappa(\hat y)|^{q}}{|\hat x-\hat y|^{1+sq}} \, d\hat x\, d\hat y = 
      \frac{|\hat e|^2}{|\tilde e|^2}\int_{\tilde e} \int_{\tilde e} \frac{|\tilde \kappa(\tilde x)-\tilde \kappa(\tilde y)|^{q}}{|B_{\hat T}(\tilde x-\tilde y)|^{1+sq}}\, d\tilde x\,d \tilde y\\
        &\ge C |\tilde \kappa|^q_{W^{s,k}(\tilde e)}.
  \end{split}
 \end{equation*} 
Because
\begin{equation*}
  \begin{split}
      \|\hat \kappa\|^q_{L^{q}(\hat e)} & = \frac{|\hat e|}{|\tilde e|}\|\tilde \kappa\|_{L^q(\tilde e)}^q
      \ge C\|\tilde \kappa\|_{L^q(\tilde e)}^q,
  \end{split}
 \end{equation*} 
we conclude $\|\hat \kappa\|_{W^{s,q}(\hat e)}\ge C_3\|\tilde \kappa\|_{W^{s,q}(\tilde e)}$.

\item[(iii)] The proof of \eqref{equ:sobolev} and \eqref{equ:psobolev}
follows the same arguments as (ii); we omit the details.
 \end{itemize}
\end{proof}

Now we are ready to prove Lemma \ref{lem:trace}.
\begin{proof}[Proof of Lemma \ref{lem:trace}.]

Let $\hat \bv\in \bH^{1/2+\delta}(\hat T)$
and set $\tilde \bv = \hat \bv\circ \phi_{\hat T}$.
Then $\tilde \bv\in \bH^{1/2+\delta}(\tilde T)$ by Proposition \ref{prop:frac}. Setting $s=\delta-1+3/p$ then we see that since $p < \frac{2}{1-\delta}$ we have $s>1/p$. Hence by the trace inequality \eqref{traceinq} we have
\[  
\|\tilde \bv\|_{L^p(\p \tilde T)} \le \| \tilde \bv\|_{W^{s-1/p,p}(\p \tilde T)} 
\le C\| \tilde \bv\|_{W^{s,p}(\tilde T)}. 
\]
By our choice we have $s-3/p=1/2+\delta-3/2$ and hence by the Sobolev inequality \eqref{sobolevembedding} 
\[  
\| \tilde \bv\|_{W^{s,p}(\tilde T)} \le C\| \tilde \bv\|_{H^{1/2+\delta}(\tilde T)}. 
\]
A change of variables along with \eqref{equ:frac} then shows
 \SG{\begin{equation*}
     \begin{split}
         \| \hat \bv\|_{L^p(\p \hat T)} &
         \le C \|\tilde \bv\|_{L^p(\p \tilde T)}\le C\| \tilde \bv\|_{H^{1/2+\delta}(\tilde T)}
         \le C \|\hat \bv\|_{H^{1/2+\delta} (\hat T)}.
     \end{split}
 \end{equation*}}
\end{proof}

\subsection{Proof of Lemma  \ref{lem:lift}}
We start with the same result for the reference element $\tilde T$.  
\begin{lemma} \label{lem:lift0}
Let $\tilde e\in \Delta_1(\tilde T)$
and $\tilde F\in \Delta_2(\tilde T)$ be an edge and face
of the reference tetrahedron, respectively, such that $\tilde e \in \Delta_1(\tilde F)$.
Then for $1< q<2$, there exists $\tilde E: \pol_{r-3}(\tilde e) \rightarrow W^{1,q}(\tilde T)$  such that $(\tilde E \tilde \kappa)|_{\tilde e} = \tilde \kappa$, $(\tilde E \tilde \kappa)|_{\p \tilde F \backslash \tilde e} = 0$, $(\tilde E \tilde \kappa)|_{\p \tilde T \backslash \tilde F} = 0$  for all $\tilde \kappa\in \pol_{r-3}(\tilde e)$,
 and the following estimates hold:
 \begin{align*}
 \|\tilde E\tilde \kappa \|_{W^{1,q}(\tilde F)} 
 &\le C \|\tilde \kappa\|_{W^{1-1/q,q}(\tilde e)},\\
 \|\tilde E\tilde \kappa\|_{W^{1,q}(\tilde T)} 
 &\le C \|\tilde \kappa\|_{W^{1-1/q,q}(\tilde e)}.
\end{align*}
\end{lemma}
\begin{proof}
We first extend $\tilde \kappa \in \pol_{r-3}(\tilde e) \subset W^{1-1/q,q}(\tilde e)$ by zero to $\p \tilde F$,  denoted by $\tilde \kappa_1$. With $q<2$ and the definition of $W^{1-1/q,q}(\tilde e)$, we have
\begin{equation*}
    \begin{split}
        \|\tilde \kappa_1\|_{W^{1-1/q,q}(\p \tilde F)}^q & = \|\tilde \kappa\|_{W^{1-1/q,q}(\tilde e)}^q + 2 \int_{e}\int_{\p \tilde F \backslash \tilde e} \frac{|\tilde \kappa(\tilde x)|^q}{|\tilde x-\tilde y|^q}\, \tilde y\, d \tilde x  \\
        & \le \|\tilde \kappa\|_{W^{1-1/q,q}(\tilde e)}^q + 2\|\tilde \kappa\|_{L^{\infty}(\tilde e)}^q \int_{\tilde e}\int_{\p \tilde F \backslash \tilde e} \frac{1}{|\tilde x-\tilde y|^q}\, d\tilde y\, d \tilde x  \\
        & \le \|\tilde \kappa\|_{W^{1-1/q,q}(\tilde e)}^q + C \|\tilde \kappa\|_{L^{\infty}(\tilde e)}^q \le C\|\tilde \kappa\|_{W^{1-1/q,q}(\tilde e)}^q,
    \end{split}
\end{equation*}
where we used that $\int_{\tilde e}\int_{\p \tilde F \backslash \tilde e} \frac{1}{|\tilde x-\tilde y|^q}\, d \tilde y\, d\tilde x$ is finite for $q<2$ and \eqref{inv1}. 

We extend $\tilde \kappa_1$ to $\tilde F$ using \eqref{righttraceinq}  and denote the extension by $\tilde \kappa_2 \in W^{1,q}(\tilde F)$ with the estimate:
\begin{equation}
    \|\tilde \kappa_2\|_{W^{1,q}(\tilde F)} \le C\|\tilde \kappa_1\|_{W^{1-1/q,q}(\p \tilde F)} \le C \|\tilde \kappa\|_{W^{1-1/q,q}(\tilde e)}.
\end{equation}
Similarly, we extend $\tilde{\kappa}_2 \in W^{1,q}(\tilde F)$ by zero to $\p \tilde T$, which 
we denote by $\tilde \kappa_3$. Set $s=2/(2-q)$ then we see that using a Sobolev inequality $\|\tilde{\kappa}_2\|_{L^{qs}(\tilde F)}\le \SG{C} \|\tilde{\kappa}_2\|_{W^{1,q}(\tilde F)}$. 
Using definition of $W^{1-1/q,q}(\p \tilde T)$ 
and H\"older's inequality, we have:
\begin{equation*}
    \begin{split}
        \|\tilde \kappa_3\|_{W^{1-1/q,q}(\p \tilde T)}^q & = \|\tilde{\kappa}_2\|_{W^{1-1/q,q}(\tilde F)}^q + \int_{\tilde F}\int_{\p \tilde T \backslash \tilde F} \frac{|\tilde{\kappa}_2(\tilde x)|^q}{|\tilde x-\tilde y|^{q+1}}\, dA(y) dA(x) \\
        & \le \|\tilde{\kappa}_2\|_{W^{1-1/q,q}(\tilde F)}^q + \|\tilde{\kappa}_2\|_{L^{qs}(\tilde F)}^q \Big(\int_{\tilde F}\int_{\p \tilde T \backslash  \tilde F} \frac{1}{|x-y|^{(q+1)s'}}\, dA(y) dA(x) \Big)^{1/s'} \\
        & \le \|\tilde{\kappa}_2\|_{W^{1-1/q,q}(\tilde F)}^q + C \|\tilde{\kappa}_2\|_{L^{qs}(\tilde F)}^q \le C\|\tilde{\kappa}_2\|_{W^{1,q}(\tilde F)}^q,
    \end{split}
\end{equation*}
where we used \SG{$s'=2/q$} which implies $(q+1)s'=(q+1) \frac{2}{q} < 3 <4$ and hence the double integral is finite. 
 
Again we use \eqref{righttraceinq} to lift $\tilde \kappa_3$ to $\tilde T$ where we denote the lifting by $\tilde E \tilde \kappa \in W^{1,q}(\tilde T)$ and it has the estimate
\begin{equation*}
    \|\tilde E \tilde \kappa \|_{W^{1,q}(\tilde T)} \le C\|\tilde \kappa_3\|_{W^{1-1/q,q}(\p \tilde T)} \le C\|\tilde \kappa_2\|_{W^{1,q}(\tilde F)}\le C \|\tilde \kappa\|_{W^{1-1/q,q}(\tilde e)}.
\end{equation*}
 
Furthermore, we have
\[
 \|\tilde E \tilde \kappa\|_{W^{1,q}(\tilde F)} = \|\tilde{\kappa}_2\|_{W^{1,q}(\tilde F)} \le C \|\tilde \kappa\|_{W^{1-1/q,q}(\tilde e)}.
\]
\end{proof}

Using  Lemma \ref{lem:lift0} and Proposition \ref{prop:frac},
we now prove 
Lemma \ref{lem:lift}.
\begin{proof}[Proof of Lemma \ref{lem:lift}.]
Let $\hat \kappa\in \pol_{r-3}(\hat e)$,
and let $\tilde \kappa \in \pol_{r-3}(\tilde e)$ (with $\tilde e = \phi_{\hat T}^{-1}(\hat e)$)
be given as $\tilde \kappa = \hat \kappa \circ \phi_{\hat T}$.
By (ii) of Proposition \ref{prop:frac} there holds  
\[ \|\hat \kappa\|_{W^{1-1/q,q}(e)} \ge C_3 \|\tilde \kappa\|_{W^{1-1/q,q}(\tilde e)}.
\]
By Lemma \ref{lem:lift0}, there exists $\tilde E: \pol_{r-3}(\tilde e) \rightarrow W^{1,q}(\tilde T)$ such that $(\tilde E \tilde \kappa)|_{\tilde e} = \tilde \kappa$ and  $(\tilde E \tilde \kappa)|_{\p \tilde F \backslash \tilde e} = 0$, $(\tilde E \tilde \kappa)|_{\p \tilde T \backslash \tilde F} = 0$. Let $E:\pol_{r-3}(e) \rightarrow W^{1,q}(\hat T)$ be defined by $(E \hat \kappa) = (\tilde E \tilde \kappa)\circ \phi_{\hat T}$. Then $E \hat \kappa|_{\hat e} = \hat \kappa$ $(E \hat \kappa)|_{\p \hat F \backslash \hat e} = 0$ , $( E \hat \kappa)|_{\p \hat T \backslash \hat F} = 0$ and with (iii) of Proposition \ref{prop:frac}, we have

\begin{equation*}
        \|E \hat \kappa\|_{W^{1,q}(\hat F)} \le C \|\tilde E \tilde \kappa\|_{W^{1,q}(\tilde F)} \le C \|\tilde \kappa\|_{W^{1-1/q,q}(\tilde e)} \le C
        \|\hat \kappa\|_{W^{1-1/q,q}(\hat e)},
\end{equation*}
and 
\begin{equation*}
        \|E \hat \kappa\|_{W^{1,q}(\hat T)} \le C\|\tilde E \tilde \kappa\|_{W^{1,q}(\tilde T)} \le C \|\tilde \kappa\|_{W^{1-1/q,q}(\tilde e)} \le C
        \|\hat \kappa\|_{W^{1-1/q,q}(\hat e)}.
\end{equation*}
\end{proof}

\end{document}